\documentclass[12pt]{amsart}
\usepackage{geometry}
\usepackage{amsmath, amssymb}
\usepackage{amscd}
\usepackage{enumitem}
\usepackage{comment}
\usepackage[english]{babel}
\usepackage{graphicx}
\usepackage{epstopdf}
\usepackage{amsaddr}
\numberwithin{equation}{section}

\newtheorem{thm}{Theorem}[section]
\newtheorem{lem}[thm]{Lemma}
\newtheorem{prop}[thm]{Proposition}

\newtheorem*{thma}{Theorem A}
\newtheorem*{thmb}{Theorem B}
\newtheorem*{thmc}{Theorem C}
\newtheorem*{thma1}{Theorem A.1}
\newtheorem*{thma2}{Theorem A.2}
\newtheorem*{thma11}{Theorem A.1.1}
\newtheorem*{thmc1}{Theorem C.1}

\theoremstyle{definition}
\newtheorem{defn}[thm]{Definition}

\theoremstyle{remark}
\newtheorem{rem}[thm]{Remark}

\begin{document}

\begin{abstract}
Let $M$ be any compact four-dimensional PL-manifold 
with or without boundary (e.g. the four-dimensional sphere or ball).
Consider the space $T(M)$ of all simplicial isomorphism classes of triangulations of $M$ endowed with the metric defined as follows: the distance between a pair of triangulations is the minimal number of bistellar transformations required to transform one of the triangulations into the other. Our main result is the existence of 
an absolute constant $C>1$ such that for every $m$ and all sufficiently large $N$ there
exist more than $C^N$ triangulations of $M$ with at most $N$ simplices such
that pairwise distances between them are greater than
$2^{2^{\ldots^{2^N}}}$ ($m$ times).

This result follows from a similar result for the space of all balanced presentations of the trivial group. (``Balanced" means that the number of generators
equals
to the number of relations). This space is endowed with the metric defined as the minimal number of Tietze transformations between finite presentations.
We prove a similar exponential lower bound for the number of balanced presentations of length $\leq N$
with four generators that are pairwise
$2^{2^{\ldots^{2^N}}}$-far from each other.
If one does not fix the number of generators, then we establish a super-exponential lower bound $N^{const\ N}$ for the number of balanced presentations
of length $\leq N$ that are
$2^{2^{\ldots^{2^N}}}$-far from each other. \end{abstract}

\title[Sizes of spaces of triangulations of $4$-manifolds]{Sizes of spaces of triangulations of 4-manifolds and balanced presentations of the trivial group.}

\author {Boris Lishak} 
\address{The University of Sydney}
\email{boris.lishak@sydney.edu.au}

\author{Alexander Nabutovsky}
\address{University of Toronto}
\email{alex@math.toronto.edu}

\maketitle

\section{Main Results}
In this paper we prove results about balanced presentations of the trivial group (Theorem B, Theorem C), triangulations of compact PL $4$-manifolds (Theorem A), Riemannian metrics subject to some restrictions on some compact smooth $4$-manifolds (Theorem A.1, Theorem A.2, Theorem A.1.1), and contractible $2$-complexes (Theorem C.1). All of these theorems imply that the spaces of corresponding structures are large: we find an exponential, or in some cases super-exponential number of presentations (correspondingly triangulations, metrics, $2$-complexes) which are extremely pairwise distant in some natural metrics. Exponentially many here means as a function of the length of presentations, number of simplices, etc.

The geometric theorems follow from the ``group" theorems B and C, which are a development of results from \cite{lishak}, \cite{bridson3}. These papers contain independent and different constructions of infinite sequences of balanced presentations of the trivial group that are very distant from trivial presentations. Here we combine the techniques from \cite{lishak} and \cite{bridson3} as well as some ideas of Collins \cite{collins} to produce an exponential and super-exponential number of pairwise distant presentations. (The growth is exponential in length, when the number of generators is fixed, and super-exponential , if it is arbitrary.)  From a group-theoretic perspective the main technical novelty of the present paper is that here we are forced to treat balanced presentations of the trivial
group that are very distant from trivial presentations as group-like objects, introduce
concepts of  homomorphisms and isomorphisms between these objects, and learn to prove that
they are not isomorphic (when this is the case).

The ``A" theorems are similar to results in dimension greater than $4$ obtained in \cite{nabutovsky2},\cite{nabutovsky3},\cite{nabutovsky3},\cite{nabutovsky4},\cite{weinberger} using different group-theoretic techniques. In these dimensions it was possible to obtain an even stronger analogue of Theorem A.2 (see \cite{nabwein1},\cite{nabwein2},\cite{weinberger},\cite{nabutovsky6}), without the topological assumption on the manifold and for a wider class of Riemannian metrics. It's not clear if it is possible to implement the techniques of this paper to prove such a generalization. We will give the precise statements now.

Let $M$ be a PL-manifold. Here {\it a triangulation} of $M$ is a simplicial complex $K$ such that its geometric realization $\vert K\vert$ is PL-homeomorphic to $M$.
We do not distinguish between simplicially isomorphic triangulations and regard them as identical. Thus, the set $T(M)$ of all triangulations
of $M$ is discrete, and for each $N$ its subset $T_N(M)$ that includes all triangulations with $\leq N$ simplices is finite. It is easy to
see that the cardinality of $T_N(M)$ is at most $N^{cN}$ for some $c=c(M)$. It is a major unsolved problem (cf. \cite{fro}, \cite{grquestions}, \cite{nabutovsky5}) if this cardinality can
be majorized by an exponential function $c^N$. The set $T(M)$ can be endowed with a natural metric defined as the minimal number
of bistellar transformations (a.k.a. Pachner moves) required to transform one of the two triangulations to the other. Recall that a bistellar transformation
is a local operation on triangulations preserving their PL-homeomorphism class. These operations can be described as follows:
One chooses a subcomplex $C$ of the triangulation formed by $i$
adjacent $n$-dimensional simplices, $i\in\{1,\ldots, n+1\}$, that is simplicially isomorphic to a subcomplex of the boundary $\partial \Delta^{n+1}$
of the standard $(n+1)$-dimensional simplex. Then one removes $C$ from the triangulation and replaces it be the complementary
subcomplex $\partial \Delta^{n+1}\setminus C$ attaching it along the boundary of $C$. Pachner proved that each pair of PL-homeomorphic finite simplicial
complexes can be transformed one into the other by means of a finite sequence of bistellar transformations, so this distance is always finite.

In \cite{nabutovsky4} it was proven that for each $n>4$ and each computable function $f:{\bf N}\longrightarrow {\bf N}$ there exists $c(n)>1$ such that for each closed $n$-dimensional manifold $M^n$ and each sufficiently large $N$ there exist at least $c(n)^N$
distinct triangulations of $M$ with $\leq N$ simplices such that all pairwise distances in $T(M^n)$ between these triangulations are greater than $f(N)$.
This result was shown to be also true for some closed $4$-dimensional manifolds, namely those that can be represented as a connected sum of any closed
$PL$-manifold and a certain number $k$ of copies of $S^2\times S^2$, where according to \cite{stan1} one can take $k=14$. However, it is desirable to know
whether or not this result holds for all closed four-dimensional PL-manifolds including $S^4$. One motivation is a connection with Hartle-Hawking model
of Quantum Gravity as well as other related models of Quantum Gravity such as Euclidean Simplicial Gravity (cf. \cite{hh}, \cite{quantum}, \cite{nabutovsky5}). Another motivation
is a natural desire to know if $T_N(S^4)$ is as ``large" as $T_N(M)$ for more complicated $4$-manifolds or the lack of topology somehow makes
the spaces of triangulations smaller. 
The main result of our paper is that a slightly weaker version of this result holds for all closed $4$-dimensional manifolds as well as compact $4$-dimensional manifolds
with boundary. To state it define functions $\exp_m(x)$ by formulae $\exp_0(x)=x,\ \exp_{m+1}(x)=2^{\exp_m(x)}.$

\begin{thma}There exists $C>1$ such that for each compact $4$-dimensional manifold $M$ (with or without boundary) and each positive integer $m$ for all sufficiently large $N$ there exist 
more than $C^N$ triangulations of $M$ with less than $N$ simplices such that the distance between each pair of these triangulations is at least $\exp_m(N)$.
\end{thma}
Note that this result follows from its particular case when $M=S^4$. Indeed, if one has $C^N$ distant triangulations of $S^4$, once can form the connected
sums of all these triangulations with {\it a fixed} triangulation of $M$.
Also, note that the same construction (exactly as in higher dimensions) implies
the following Riemannian analog of Theorem A: 

\begin{thma1} There exist positive constants $C>1$, $c_0$, $c_1$, $c_2$, $const$ such that for each closed $4$-dimensional manifold $M$
and each non-negative integer $m$ for each
sufficiently large $x$ there exist more than $\exp(x^{c_0}) $ Riemannian metrics on $M$ such that 1) each of those metrics has sectional curvature between $-1$ and $1$,
injectivity radius greater than $c_1$, volume greater than $c_2$ but less than $x$, and diameter less than $const\ln x$;
2) For each positive $c_3$ there is no sequence of jumps of length $\leq c_3\exp_{m+1}(-c_2\ln x)$
in the Gromov-Hausdorff metric that connects a pair of these metrics
within the space of Riemannian metrics on $M$ with sectional curvature between $-1$ and $1$, volume $>c_3$ and diameter $\leq \exp_m(const \ln x)$.
(If $m=0$, the upper bound for the length of jumps becomes $c_3x^{-c_2}$ and the upper bound
for the diameter is the same $const\ln x$ as in the condition 1.)

\end{thma1}
If the Euler characteristic of $M$ is non-zero and its sectional curvature satisfies $\vert K\vert\leq 1$, 
then the Gauss-Bonnet theorem implies that $vol(M)>const>0$.
Therefore, Theorem A.1 implies the existence of exponentially many connected components of the sublevel set of the diameter
functional on the space of isometry classes of Riemannian metrics on $M$ with $\vert K\vert \leq 1$. 
Moreover, each pair of these components can merge only in a connected
component of a sublevel set of diameter only for a much larger value of $x$.
A natural idea will be to look for minima of diameter on connected components of its sublevel sets, because the minimum of any
continuous functional on a connected component of its sublevel set will be automatically a local minimum of this functional.
In order to ensure the existence of the minimum of diameter on all connected components of its sublevel set we are going first
to somewhat enlarge the considered space:
Denote the closure of the space of Riemannian structures (i.e. isometry classes of Riemannian metrics)
on $M$ with $\vert K\vert\leq 1$ in the Gromov-Hausdorff topology by
$Al_1(M)$. Elements of $Al_1(M)$ are Alexandrov spaces with curvature bounded from both sides.
They are isometry classes of $C^{1,\alpha}$-smooth metrics on $M$, and the sectional curvature
can be defined a.e. Consider
the diameter as a functional on $Al_1(M)$. Its minima in connected components
of sublevel sets are also local minima on the whole space. Therefore,
Theorem A.1 implies that the diameter has ``many" local minima on 
$Al_1(M)$, and at least some of these minima must be very deep. 
The number of these minima is at least exponential in a positive power of $x$,  and $x$ behaves as the exponential
of our upper bound for the diameter. Therefore, denoting our upper bound for the diameter by $y$, and observing that $x>\exp ({y\over const})$,
we see that Theorem A.1 implies the following theorem:

\begin{thma2} There exists a positive constant $c$  such that for each closed $4$-dimensional Riemannian manifold
$M$ with non-zero Euler characteristic and each positive
integer number $m$ the diameter regarded as a functional on $Al_1(M)$ has infinitely
many distinct local minima $\mu_i$ such that 1) the sequence $diam(\mu_i)$ is an unbounded increasing
sequence; 2) the number of $i$ such that $diam(\mu_i)\leq y$ is greater than $\exp(\exp(cy))$;
3) Each path or a sequence of sufficiently short jumps in $Al_1(M)$ that starts at $\mu_i$ and ends
at a point with a smaller value of the diameter must pass through a point where the value of
the diameter is greater than $\exp_m(y)$.
  
\end{thma2}

For $n\geq 5$ the existence part of these results first appeared in \cite{nabutovsky3} (Theorems 9, 11). Later \cite{nabwein1} (see also \cite{nabwein3})
the same and even stronger results were proven without the assumption that the Euler characteristic of $M^n$
does not vanish. (More importantly, the techniques of \cite{nabwein1} can be applied to other Riemannian
functionals, for example, to diameter regarded as a functional on the space $al_1(M^n)$ of
Alexandrov structures on $M^n$ with curvature $\geq -1$.)
The depths of local minima grow not only faster than any finite tower of exponentials but faster than any computable function.
Also, it was proven in \cite{nabwein1} that the
distribution function for ``deep'' local minima of $diam$ on $Al_1(M^n)$ grows at least exponentially.
Shmuel Weinberger observed that the distribution function for the ``deep" local minima of diameter is, in fact,
doubly exponential (\cite{weinberger}, Theorem 1 on p. 128). For $n=4$ the corresponding existence results were proven in \cite{ln}.

Note that from Theorem A.1 almost immediately follows the following result with a somewhat nicer statement:

\begin{thma11}
There exist positive constants $c_0, c_1, const$ and $C>1$ such that for each $m$ and each closed $4$-dimensional Riemannian manifold $M$ with non-zero Euler characteristic
for all sufficiently large $x$ there exists at least $\exp(x^{c_0})$ Riemannian metrics $g_i$ on $M$ such that Riemannian manifolds $M_i=(M, g_i)$
have volume $\leq x$, injectivity radius greater than $c_1$ and diameter less than $const\ \ln x$ with the following
property: Let $i\not =j$, and $f:M\longrightarrow M$ be any diffeomorphism. If we consider it as a map
$f_{ij}$ between Riemannian manifolds $M_i$ and $M_j$, then either $\sup_{x\in M_i}\vert Df_{ij}(x)\vert$ or
$\sup_{x\in M_j}\vert Df^{-1}_{ij}(x)\vert$ will be greater than $\exp_m(x)$.
\end{thma11}

To see that Theorem A.1 follows from Theorem A.1.1, we can use the same Riemannian metrics on $M$ for both theorems.
A well-known fact which is a part of all proofs of the Gromov-Cheeger compactness theorem is that  two sufficiently
Gromov-Hausdorff close Riemannian manifolds satisfying the conditions of Theorem A.1 (or Theorem A.1.1) are diffeomorphic.
Moreover, the proofs yield concrete upper bounds for the the Lipschitz constants. Assuming that the condition 2) of Theorem A.1
does not hold we can multiply the upper bounds for the  Lipschitz constants for the diffeomorphisms
corresponding to small jumps in the condition 2 and obtain "controlled" upper bounds for the composite diffeomorphisms, which contradicts Theorem A.1.1.

\par
Theorems A, A.1, A.1.1 follow (see the last section) from the following theorem about balanced presentations. The {\it length} of a finite presentation of a group is defined 
as the sum of the lengths of all relators plus the number of generators. 

\begin{thmb} There exists a constant $C>1$ such that for each $m$ and all sufficiently large $N$ there exist more than $C^N$ balanced presentations of 
the trivial group of length $\leq N$ with four generators and four relators such that for each pair of these presentations
one requires more than $exp_m(N)$ Tietze transformation in order to transform one of these presentations into the other.

\end{thmb}

If one does not restrict the number of generators and relators, then there is  a superexponentially growing number of
pairwise distant balanced presentations of the trivial group:

\begin{thmc}

There exists a constant $c>0$ such that for each $m$ and all sufficiently large $N$ there exist more than $N^{cN}$ balanced presentations of 
the trivial group of length $\leq N$ such that for each pair of these presentations
one requires more than $exp_m(N)$ Tietze transformation in order to transform one of these presentations into the other. Moreover, these balanced
presentations can be chosen so that the length of each relation is equal to $2$ or $3$,
and each generator appears in at most $3$ relations.
\end{thmc}

If we take one of the balanced presentations in Theorems B or C, and consider its presentation complex, i.e. the $2$-complex with one $0$-cell,  $1$-cells corresponding to the generators of the presentation, and $2$-cells corresponding to its relations, then we are going to obtain
a family of contractible $2$-complexes. In the situation of Theorem B these complexes will have $4$ $1$-cells and $4$ $2$-cells attached along
words of length $<N$; in the situation of Theorem C the number of $1$-cells or $2$-cells will be $<N$, and each $2$-cell will be either a digon
or a triangle. In the situation of Theorem B the number of these $2$-complexes is at least $C^N$, in the situation of Theorem C $N^{cN}$. 
We can endow each of these complexes by a path metric obtained by considering each $1$-cell as a circle of length $1$, and each $2$-cell
as the $2$-disc with the length of circumference equal to the length of the corresponding relator.
Also, we can subdivide these complexes into simplicial complexes with $O(N)$ simplices by subdividing each $2$-cell into
triangles in an obvious way.
Now given a pair of such contractible $2$-complexes $K_1, K_2$ we can ask for specific Lipschitz maps (=homotopy equivalences )
$f:K_1\longrightarrow K_2$, $g:K_2\longrightarrow K_1$ and Lipschitz homotopies $H_1$ between $f\circ g$ and the identity map of 
$K_2$, and $H_2$ between $g\circ f$ and the identity map of $K_1$ such that the maximum of Lipschitz constants of $f,g, H_1, H_2$
is minimal possible.
Theorems B and C imply that for all sufficiently large $N$ 
the Lipschitz constant of either $f$, $g$ or at least one of the two homotopies must be greater than $\exp_m(N)$. (On the other hand $f$ and $g$ can be chosen
as constant maps, so they do not need to have large Lipschitz constants.) 
The proof of this fact will not be presented here and is similar to the proofs of Theorems A and A.1 given below.

Alternatively, we can take the contractible simplicial $2$-complexes $K_i$ constructed by triangulating presentation
complexes of presentations from Theorem C, consider all possible subdivisions $K_1'$ of $K_1$, $K_2'$ of $K_2$,
simplicial maps $f:K'_1\longrightarrow K_2,g:K_2'\longrightarrow K_1$, define $F_1=g\circ f$ and $F_2=f\circ g$,
and, finally, consider all homotopies 
$H_i$ between $F_i$ and $id_i:K_i\longrightarrow K_i$ that are simplicial maps defined on some simplicial subdivisions of $K_i'\times [0,1]$ with values in $K_i$  for $i=1,2$. 
(Observe, that, in general,  $F_i$ are not simplicial maps from $K_i'$ to $K_i$, $i=1,2$. However,
we can subdivide each simplex $s$ of $K_1'$ into simplices $f^{-1}(d)$, where $d$ runs over all simplices of $K_2'$ in
$f(s)$ (and, if $dim\ s=2$ and $dim\ f(s)=1$, then we might need to further subdivide quadrilaterals $f^{-1}(d)$ into pairs of triangles).
After subdividing $K_2'$ in a similar way,  we will be able to regard $F_i$ as simplicial maps of
the constructed subdivisions $K_i''$ of $K_i'$ into $K_i$. Now the relative simplicial approximation theorem would
imply the existence of some subdivisions $\bar K_i$ of $K_i''\times [0,1]$ and simplicial homotopies $H_i:\bar K_i\longrightarrow K_i$ between $F_i$ and $id_i$.)

Now we can define the complexity of the
quadruple of maps $f, g, H_1, H_2$ as the total number of $3$-simplices in the subdivisions $\bar K_i$ of $K_i'\times [0,1], i=1,2$, used to define $H_1, H_2$. Define the {\it witness complexity
of homotopy equivalence of $K_1, K_2$} as the minimum of the complexity over all such quadruples $f,g, H_1, H_2$ of simplicial maps.
Theorem C easily implies that:

\begin{thmc1}
For some $c>1$ and each $m$ for all sufficiently large $N$ there exist more than $N^{cN}$
contractible $2$-dimensional simplicial complexes $K_l$ with at most $N$ $2$-simplices such that for each pair $K_j, K_k$, ($j\not =k$), the witness
complexity of homotopy equivalence of $K_j$ and $K_k$ is at least $\exp_m(N)$.

\end{thmc1}

Theorem C can be straightforwardly obtained from Theorem B for $\lfloor N\ln N\rfloor$ just by
 rewriting the finite presentations in an appropriate way. Theorem C.1 will be proven in the last section, and easily follows
 from our proof of Theorem C.
Theorem A will be deduced from a modified version of Theorem C explained
in section 6. One applies this theorem for $N_1$ such that $N$
in Theorem A satisfies $N=\lfloor N_1\ln N_1\rfloor$.
We realize the balanced finite presentations of the trivial group from Theorem C as ``apparent" finite presentations
of PL $4$-spheres triangulated with $O(N)$ simplices. More precisely,
we start from the connected sum of several copies of $S^1\times S^3$ (one copy for each generator) , realize relators by embedded circles and kill
them by surgeries. Then we demonstrate that the resulting manifold can be triangulated in $O(N)$ simplices. 
The ``balanced" condition ensures that the resulting manifold will be a homology sphere. The triviality of the group implies that it
is a homotopy sphere and, thus, by the celebrated theorem of M. Freedman homeomorphic to $S^4$. But all our presentations obviously satisfy the Andrews-Curtis
conjecture and, therefore, the resulting manifolds will be PL- (or smooth) spheres. 

Of course, this construction can be also performed in all dimensions greater than four.
However, for higher dimensions one has a much larger stock of suitable finite presentations of the trivial group (see \cite{nabutovsky4}): 
One can start from a sequence of finite presentations used in the proof of S. Novikov theorem asserting the non-existence of the algorithm recognizing
$S^n$ for each $n>4$. There is no algorithm deciding which of these presentations are presentations of the trivial group, which means that the number
of Tietze transformations required to transform presentations of the trivial group that appear in this sequence to the trivial presentation is not bounded by any computable
function. On the other hand, all these presentations are presentations of groups with ``obviously" trivial first and second homology groups. As a result,
one can use the Kervaire construction (\cite{kervaire2}) to realize them as ``apparent" finite presentations of $n$-dimensional homology spheres which are diffeomorphic
to $S^n$ (as the groups are trivial). However, there is no easy way to see that the homology spheres are diffeomorphic to $S^n$, as otherwise we would be able to see that seed finite presentations are
finite presentations of the trivial group. The starting sequence of finite presentations codes a halting problem for a fixed Turing machine, and it was
noticed in \cite{nabutovsky4} that one can use the concept of time-bounded Kolmogorov complexity and a theorem proven by Barzdin to choose this Turing machine so that
it is possible to ensure 
the existence of not only some triangulations of $S^n$ far from the standard triangulation but the existence of an exponentially growing number of such triangulations (as in Theorem A).

The four-dimensional situation is different from the higher-dimensional case.
J.P. Hausmann and S. Weinberger proved that the vanishing of the first two homology groups of a finitely presented
group is no longer sufficient to realize this group as the fundamental group of a $4$-dimensional homology sphere (\cite{hauswein}).
(It was the main result of \cite{kervaire2} that for each $n>4$ $H_1(G)=H_2(G)=\{0\}$ is the necessary and sufficient condition for the existence of a smooth $n$-dimensional
homology sphere with fundamental group $G$.) The only known general sufficient condition of realizability of a group given by a finite presentation
as the fundamental group of a $4$-dimensional homology sphere is that this finite presentation is balanced. The condition that the number of relators
is equal to the number of generators is a very strong condition that seemingly precludes coding of Turing machines in such finite presentations. In fact, it
is a famous unsolved problem whether or not there is an algorithm that decides if a group given by a balanced presentation is trivial.
However, the example of the Baumslag-Gersten group $B=\langle x,t\vert x^{x^t}=x^2 \rangle$ that has Dehn function growing faster than a tower of exponentials of
height $[\log_2 n]-const$ suggests an idea of adding the second relation $w_n=t$, where $w_n$ runs over words representing the trivial element
in $B$ that have very large areas of their van Kampen diagrams. (Before going further
recall that $a^b$ means $b^{-1}ab$, and observe that $B$ can also be written as 
$\langle x,y,t\vert x^y=x^2; x^t=y\rangle$, which shows that it can be obtained from the infinite cyclic group by performing two HNN-extensions with
associated subgroups isomorphic to $\mathbb{Z}$.)
As a result, we obtain balanced presentations of the trivial group with two generators, such that a very large number
of Tietze transformations is required to transform these presentations to the trivial presentation provided that
one proceeds in the most obvious way (namely,  first ``proving" that $w_n=e$ using only the first relation, then concluding that $t=e$ from the second
relation.)
Thus, it is natural to conjecture that these balanced presentations of the trivial group
will be very far from the trivial presentation and, therefore, can be used to obtain triangulations of $S^4$ and other $4$-manifolds that are very far from the
standard ones. The second author came with this idea in 1992 and discussed it during 90's with many colleagues but was never able to prove this conjecture.
(The conjecture was mentioned  in \cite{nabutovsky5}, that first appeared as 2001 IHES preprint IHES/M/01/35.) 
The difficulty is that Gersten's proof of the fact that $B$ has a very rapidly growing Dehn function uses the fact that $B$ is non-trivial. In particular,
Gersten considers the universal covering of the presentation complex of $B$, and this cannot be done for presentations of the trivial group. On the other hand,
it seems that it is difficult to apply methods based on considerations of van Kampen diagrams when one needs to deal with two relators rather than one.

This problem was solved by the first author in \cite{lishak}, \cite{lishak3}.
The main new idea was use a modified version of the small cancellation theory over HNN-extensions.
(It is easy to see that $B$ can be obtained from $\mathbb{Z}$ by performing two HNN-extensions.) In \cite{lishak} he used ``long"
words $w_n$ of very special form; recently he improved his approach so that it yields the conjecture for ``most" words $w_n$ representing the trivial
element including the most natural ones (\cite{lishak3}). To get a flavour of the main idea of \cite{lishak} consider pseudogroups associated with finite presentations of a group where a word $w$ is regarded as trivial only if it can be represented as a product of not more than $f(\vert w\vert)$ conjugates of relators, where $f$ is not too rapidly growing function of the length of $w$. (This concept is similar to the notion of effective universal coverings introduced \cite{nabutovsky7}). If we consider such ``effective" pseudogroups produced from $B$ , the word $w_n$
will not represent the trivial element in a pseudogroup. Moreover, we can hope that for an appropriate choice
of $w_n$ ``effective" pseudogroups associated with $B$ with added relation $w_n=t$ satisfy a small cancellation condition and, therefore, are non-trivial
as effective pseudogroups. This implies that corresponding balanced presentations of the trivial group are very far from the trivial presentation.

Unbeknown to the authors the problem of construction of balanced presentations of the trivial group that are far from the trivial presentation was also of
interest to Martin Bridson (who also was not aware of the interest of the second author to this problem). His preprint \cite{bridson3} with another solution
appeared on arxiv two weeks after \cite{lishak} but, as we later learned, Bridson announced his solution in a series of talks in 2004-2006 and mentioned it
in his ICM-2006 talk \cite{bridson2}. We are going to describe his main ideas from our viewpoint involving the concept
of effective pseudogroups. He also starts from $B$. First, he considers a ``fake" HNN-extension $\tilde B$ of $B$ with stable letter $s$ and new relation
$sw_ns^{-1}=t$. As $w_n=e$ the resulting group is isomorphic to $\mathbb{Z}$, but in the realm of effective pseudogroups $w_n$ is non-trivial, so from
some intuitive viewpoint this can be regarded as an HNN-extension of effective pseudogroups.
Then he  takes two copies $\tilde B_1$, $\tilde B_2$ of $\tilde B$ and considers the amalgamated free product of effective pseudogroups
$\langle \tilde B_1*\tilde B_2\vert s_1=x_2, x_1=s_2 \rangle$. Obviously, these presentations are the presentations of the trivial group but in the realm
of effective pseudogroups this will be non-trivial pseudogroups. (Again, this is our interpretation of Bridson's examples, not his. 
Also, at the moment there is no theory of effective pseudogroups. Yet this point of view suggests a possible simplification
of Bridson's construction, namely, ``fake" amalgamated products of two copies of $B$ $\langle B_1*B_2\vert t_1=w_{n2}, w_{n1}=t_2\rangle$.)

The main idea of the proof of Theorem B (and, thus, all other results of the present paper) is the following. For each $n$ we produce
exponentially many versions of the fake HNN-extension $\tilde B$ of $B$. These versions involve variable words $v$ of length not exceeding $const\ n$, and words $w_k$, $k<<n$, satisfying an ``effective'' small cancellation condition as in \cite{lishak}, which gives even more control over $\tilde B$.
As the result, the finite presentations obtained from two copies of $\tilde B$ as in \cite{bridson3} for two different $v$
will represent effective pseudo-groups that are not effectively isomorphic. This implies that they cannot be transformed one into the other by
a ``small" number of Tietze transformations. Thus, our construction is a hybrid of constructions in \cite{lishak} and \cite{bridson3}.

At the moment we do not have any infrastructure for a theory of effective pseudogroups (but hope to develop it in a subsequent
paper). Therefore, the arguments involving the concept of effective pseudogroup should be regarded only as intuitive ideas useful
for understanding of our actual proofs
(that do not involve the concept of effective pseudogroups).

\par
To give a flavour of our further arguments consider $\langle B_1*B_2\vert t_1=w_{k2}, w_{k1}=vt_2v^{-1}\rangle$, where this time $w_{k1}$, $w_{k2}$
denote two copies of a fixed word $w_k$ representing a non-trivial element of infinite order in
the Baumslag-Gersten group $B$, and $v$ is a variable word in $y, yx$. 
It is natural to conjecture that two such groups defined for different $v$'s are not isomorphic. To prove this fact one can investigate what can
happen with the generators of $B_1$ under a possible isomorphism. One will be using the fact that here we are dealing with the amalgamated free products
of two copies of $B$, as well as the description of the outer automorphism group of $B$ found by A. Brunner in \cite{brunner}. Brunner proved that
$Out(B)$ is isomorphic to the additive group of dyadic rationals. (A self-contained exposition of this result intended for geometers
can be found in [Lisc].) In the present paper we will need to establish an effective version of Brunner's theorem and, then, a non-existence
of an ``effective isomorphism" between ``effective pseudogroups" $\mu_{v_1}$ and $\mu_{v_2}$, which is an informal way of saying that
$\mu_{v_1}$ and $\mu_{v_2}$ cannot be transformed one into the other by a ``short" sequence of Tietze transformations.
\par
Finally, note that it seems plausible that one can construct the desired presentations as $\langle B\vert w_k=vtv^{-1}\rangle$, where $v$ runs over
an exponentially large set of words in $y, yx$, but at the moment methods of \cite{lishak} seem insufficient to prove that
these presentations are exponentially far from each other. But, if true, this would reduce the smallest possible 
number of generators in Theorem B from $4$ to $2$. 
\par
In the next section we describe the construction of $\mu_v$. In the third section we prove some quantitative results about Baumslag-Gersten group, necessary for the fourth section, where we essentially prove that any two different presentations from our construction are not isomorphic as effective pseudogroups. In the fifth section we prove Theorem B, in the sixth we prove Theorem C. In the last section we prove Theorems A, A.1, A.1.1.

\section{Notation and the Construction} \label{ch5:construction}

Denote by $exp_n(x)$ the tower of exponents of height $n$, i.e. $exp_n$ are recursively defined by $exp_0(x) = x$, $exp_{n+1} = 2^{exp_n(x)}$. Let $E_n = exp_n(1)$. As usual, $x^y$ denotes $y^{-1}xy$, where $x, y$ can be words or group elements. Let $l(w)$ be the length of the word $w$. If $w$ represents the identity element, denote by $Area_{\mu}(w)$ the minimal number of 2-cells in a van Kampen diagram over the presentation $\mu$ with boundary cycle labeled by $w$. For a presentation $\mu$ denote by $l(\mu)$ its total length, i.e. the sum of lengths of the relators plus the number of generators.

\par

 Let $G=\langle x,y,t | x^y = x^2, x^t = y \rangle$, $G$ is called the Baumslag-Gersten group, an HNN-extension of the group $K =\langle x,y,| x^y = x^2\rangle$ (called the Baumslag-Solitar group).

\par

We are going to use the same word of large area as in \cite{lishak}.

Let $w_n$ be defined inductively as follows. Let 
\begin{displaymath}
w_{n,0} = [y^{-E_n}xy^{E_n},x^{3}]  [y^{-E_n}xy^{E_n},x^{5}]  [y^{-E_n}xy^{E_n},x^{7}] .
\end{displaymath}
Here $[a,b]$ denotes $aba^{-1}b^{-1}$,  the commutator of $a$ and $b$
Suppose $w_{n,m}$ is defined, then let $w_{n,m+1}$ be the word obtained from $w_{n,m}$ by replacing subwords $y^{\pm E_{n-m}}$ with $t^{-1} y^{-E_{n-m-1}}x^{\pm 1}y^{E_{n-m-1}} t$. Finally, let $w_n = w_{n,n}$. 

\begin{rem} \label{ch5:chapter1sum}
We can make an estimate $l(w_n) \le 100 \cdot 2^n$. In the terminology of \cite{lishak}, $w_{n,1}$ is $E_n$-reduced, i.e. any non-circular $t$-band in a digram for $w_{n,1}$ is of length at least $E_n$. This in particular implies that the area of $w_n$ is at least $E_n$.
\end{rem}

Before defining the presentations we prove a lemma.
\begin {lem} \label{ch5:solitar}
Let $v, v'$ be two different words in the alphabet $\{y, yx\}$, then $v \neq v'$ in $G$.

\end{lem}
\begin{proof}
One can see that by choosing $x$ as the coset representative for $x\langle x^2 \rangle$ and applying the theory of normal forms for an HNN-extension $K$. Alternatively, one can see that $v = y^i x^j$ in $K$, where $j$ can be represented in the binary notation as follows. The length of the number is equal to $i$. There is a digit $1$ for each $x$ in the word $v$, the digit is placed in the $n$-th position if there are $n-1$ letters $y$ to the right of that $x$. The rest of the digits are $0$. For example, $yxyyxyxy = y^5 x^j$, where $j=10110_2$. Clearly, this number is unique for a word $v$. To finish the proof we notice that if $v = y^i x^j = v' = y^m x^n$, then $i=m$ because $y$ is the stable letter of $K$, and therefore $j=n$. Finally, if $v \neq v'$ in $K$, then $v \neq v'$ in $G$, because $G$ is an HNN-extension of $K$.
\end{proof}

Let $v$ be an arbitrary word of length in $y$, $yx$ , where $y$, $yx$ appear only in non-negative degrees.
Let $H_v = \langle x,y,t,s  | x^y = x^2, x^t = y, s^{-1}v w_n^{-1} v^{-1} w_n s = t \rangle $, and define another copy of this presentation $\hat H_v  = \langle \hat x, \hat y,\hat t,\hat s  | \hat x^{\hat y} = \hat x^2, \hat x^{\hat t} = \hat y, \hat s^{-1}\hat v \hat w_n^{-1} \hat v^{-1} \hat w_n \hat s = \hat t \rangle $. Note, these are presentations of the infinite cyclic group. Let:

$$\mu_v = H_{v} \underset{s = \hat x, x = \hat s}{*} \hat H_{v}. $$

Replacing $s$ by $\hat x$, $\hat s$ by $x$, $y$ by $x^t$, $\hat y$ by $\hat x^{\hat t}$ we
can reduce the number of generators to $4$. Using the commutator notation the resulting presentations of the trivial group will look as follows:

$$\mu^0_v=\langle x, t, \hat x, \hat t | x^{x^t}=x^2, \hat x^{\hat x^{\hat t}}=\hat x^2, [v, w_n^{-1}]^{\hat x}=t, [\hat v, \hat w_n^{-1}]^x=\hat t\rangle.$$

Later we will use this construction for words $v$ of length $N$ and choose $n$ less than $\log_2 N -20$, and therefore the lengths of these presentations are $\sim N$. There are $2^N$ of them, and we will prove that they are very far from each other in the metric defined as the minimal number of Tietze transformations required to transform one presentation into the other. But for now we will treat the length of $\mu_v$ as a variable.

\begin{rem}
Results of \cite{bridson3} imply that, in particular, $Area_{\mu_v}(x) \geq Area_{G}(v w_n^{-1} v^{-1} w_n)$. We will require finer results of the same type (see the next section) to prove that not only $Area_{\mu_v}(x)$ is large, but also that for different words $v$ of the type defined in the previous lemma the presentations are ``far'' from each other. The exponential number of such words as a function of length will give us the exponential number of such presentations. In order to prove Theorem B we will be using words $v$ of length $\sim N$ and words $w_n$ with $n<<N$. Also, note that although given presentations
of $\mu_v$ have $8$ generators and relators one can use four relations in order to eliminate $4$ generators (namely, $y,s, \hat y, \hat s$) as above and rewrite these presentations as balanced presentations with $4$ generators.
\end{rem}

We introduce more notation. For $w \in \mu_v$ denote by $\hat l(w)$ the minimal $m$ such that $w=a_1...a_m$ (equality in the free group), where neighbouring $a_k$ are from different factors of $\mu_v$.
For $w \in H_{v}$ denote by $l_s(w)$ the number of letters $s, s^{-1}$ in $w$. Similarly, for $w \in \hat H_{v}$ denote by $l_ {\hat s}(w)$ the number of letters $\hat s, \hat s^{-1}$ in $w$, and for $w \in G$ denote by $l_t(w)$ the number of letters $t, t^{-1}$ in $w$.

\section{Quantitative Results about the Baumslag-Gersten Group} \label{ch5:lemmas}
As we mentioned before, this section contains technical lemmas similar in the spirit to Theorem B from \cite{bridson3}. But we will prove finer results using ``effective'' small cancellation theory from \cite{lishak}. We will use these lemmas to obtain results about ``effective'' homomorphism in the next section.
\par
Recall, that $K$ denotes the Baumslag-Solitar group, $G$ is the Baumslag-Gersten group, and $w_n$ were defined in the previous section.
\begin {lem} \label{ch5:s1}
Let $\tilde w$ be a non-empty word in $\{ a, a^{-1},b \}$ that does not contain more than $1$ letter $b$ consecutively. Let $w$ be a word obtained from $\tilde w$ by replacing $a$ with $w_n$, $a^{-1}$ with $w_n^{-1}$, and $b$ with an element $B$ of $K$, possibly different for any particular instance of $b$, satisfying the following conditions. If $b$ is between two letters $a$, then $x^{-7} B \neq y^i$, if $b$ is between $a$ and $a^{-1}$, then $x^{-7} B x^7 \neq y^i$, if $b$ is between $a^{-1}$ and $a$, then $B \neq y^i$, and if $b$ is between two letters $a^{-1}$, then $ B x^7 \neq y^i$ (not equal in $K$). Then if $w=1$ in G, $Area_G(w) \geq E_n$.
\end{lem}

\begin{proof}
Consider a van Kampen diagram for $w$. Let us call letters $t, t^{-1}$ on the boundary of this diagram ``outer'' if they come from $w_{n,1}$. Recall that $w_{n,1}$ is the product of three commutators of the form

\begin{displaymath}
[t^{-1} y^{-E_{n-1}}x^{-1}y^{E_{n-1}} t x t^{-1} y^{-E_{n-1}}x^{1}y^{E_{n-1}} t,x^3],
\end{displaymath}
where in the other two commutators $x^3$ is replaced with $x^5$ and $x^7$. First, notice that any $t$-band in the diagram originating on an outer letter has to end on an outer letter. That follows from counting the number of letters $t$ and subtracting from that the number of letters $t^{-1}$ between the two ends of a $t$-band. This difference has to be equal to $0$. Pick an outer letter and a $t$-band corresponding to it. Pick another outer letter between the ends of this $t$-band, there is the corresponding inner $t$-band. Continue until we find a $t$-band between neighbouring outer letters. We claim that this $t$-band has length $E_n$. The only possible neighbouring outer letters are $t^{-1}At$ ($A=_G x^{\pm E_n}$), $tx^it^{-1}$ (not a pinch), $tx^{-7} B x^{7} t^{-1}$ (coming from $w_n B w_n^{-1}$), $t B t^{-1}$ (coming from $w_n^{-1} B w_n$), $tx^{-7} B t^{-1}$ (coming from $w_n B w_n$), $t B x^{7} t^{-1}$ (coming from $w_n^{-1} B w_n^{-1})$. The last four pairs are not pinches because of the requirements on $B$. The claim and the lemma follow.

\end{proof}

\begin {lem} \label{ch5:s2}
Assume $Area_G(x^{i} u_1 x^j u_2) < E_n$, where $u_1,u_2$ are powers of $v w_n^{-1} v^{-1} w_n$ or $t$, $i \neq 0$ and $v$ is as in Lemma \ref{ch5:solitar}. Then $u_1,u_2 = 1$ as words. Similarly, if $u_3 = (v w_n^{-1} v^{-1} w_n)^k$ for $k \neq 0$, then $Area_G(u_3) \geq E_n$.
\end{lem}
\begin{proof}
We look at cases. If one of the words is a power of $t$ then the other has to be the inverse of this power, and since $i \neq0$,  $j \neq 0$, which implies $v_1=v_2=0$. If one of the words is a power of $v w_n^{-1} v^{-1} w_n$, then so is the other one and $i=j$. We want to apply Lemma \ref{ch5:s1} and therefore want to check the conditions in its statement. Between $w_n^{-1}$ and $w_n$ we can have $v^{-1} = x^{-m}y^{-k} \neq y^p$. Between $w_n$ and $w_n^{-1}$ we can have $v$, $x^i$, $x^i v$, $v^{-1}x^i$, or $v^{-1} x^i v$. We need to check that if we conjugate any of them by $x^7$ we won't get a power of $y$. It is true for the second and the fifth element because $i\neq 0$. For the rest we will use the following fact, if $v = y^k x^m$, then $k>0$ and $m$ is even. Checking for the first element: $x^{-7} v x^{7} = x^{-7} y^k x^m x^{7} = y^k x^{odd} \neq y^p$. Similarly for the third: $x^{-7} x^i v x^{7} = x^{-7} x^i y^k x^m x^{7} = y^k x^{odd} \neq y^p$. And the fourth: $x^{-7} v^{-1} x^i x^{7} = x^{-7} x^{-m}y^{-k} x^i x^{7} = x^{odd} y^{-k} \neq y^p$. Therefore, the lemma follows from Lemma \ref{ch5:s1}.
\end {proof}

\begin {lem} \label{ch5:s4}

Let $Area_G(g u_1 g^{-1} u_2) < E_{n-1}$, where $g$ is a word in $G$, $u_1 =A^i, u_2= B^j$, where $A,B$ are $v w_n^{-1} v^{-1} w_n$ or $t$, $i,j$ are less than $E_{n-1}$, and $v$ is as in Lemma \ref{ch5:solitar}. Then $A=B$ and $i= \pm j$. Furthermore, if $A=B=t$, then $g=_G t^k$ through area less than $E_{n-1}$.

\end{lem}
\begin {proof}
When $A$ or $B$ is $t$, the conclusion is clear. Suppose $A=B=v w_n^{-1} v^{-1} w_n$. Let $u'_1$ be $u_1$ after all $w_n$ are replaced with $w_{n,1}$. Similarly define $u'_2$. Then $Area_G(g u'_1 g^{-1} u'_2) \leq E_{n}$ (one can see the proof of Theorem 2.3 from \cite{lishak} for a complete calculation). Consider an annular van Kampen diagram of area less than $E_n$ for this conjugation. The $t$-bands on this diagram can not originate and end on the same boundary component (see the proof of Lemma \ref{ch5:s2}). Therefore the number of letters $t$ is the same for both boundary components and $i= \pm j$.

\end {proof}

Notice the decrease of the upper bound ($E_{n-1}$ versus $E_n$ in Lemma \ref{ch5:s2}) in the previous lemma. It is there for technical reasons, and we believe can be eliminated with some extra work. Similarly, the factor of $2$ in the next lemma is, probably, unnecessary.

\begin{figure}[h]
\centering
\includegraphics[scale=0.5]{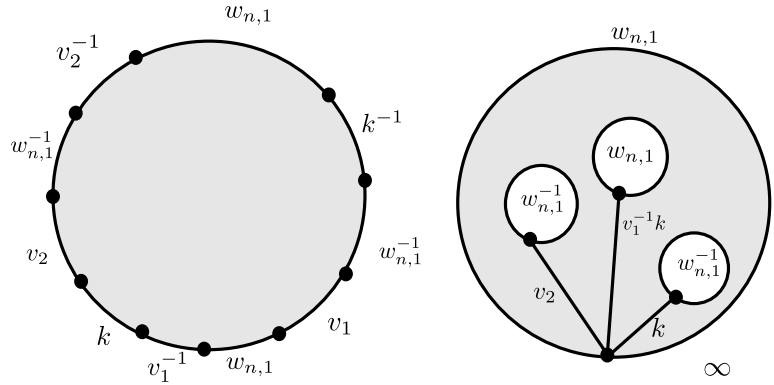}
\caption{Two Van Kampen diagrams for $k^{-1} (v_1 w_{n,1}^{-1} v_1^{-1} w_{n,1}) k = v_2 w_{n,1}^{-1} v_2^{-1} w_{n,1}$. The one on the right was obtained from the one on the left by gluing parts of the boundary together. On can think of the right diagram as spherical by placing $\infty$ outside of the outer $w_{n,1}$.}
\label{ch5:fig:sphere1}
\end{figure}

\begin {lem} \label{ch5:s5}
Let $v_1 \neq v_2$ be as in Lemma \ref{ch5:solitar}, then $Area_G( k^{-1} (v_1 w_n^{-1} v_1^{-1} w_n) k ( v_2 w_n^{-1} v_2^{-1} w_n)^{-1}) \geq \frac {1} {2} E_n$, where $k$ is a word in $G$.

\end{lem}
\begin {proof}

We can apply less than $100E_{n-1}$ relations to convert $w_n$ to $w_{n,1}$, obtaining a van Kampen diagram for $k^{-1} (v_1 w_{n,1}^{-1} v_1^{-1} w_{n,1}) k = v_2 w_{n,1}^{-1} v_2^{-1} w_{n,1}$ of area less than $E_n$. We want to apply our ``effective'' small cancellation theory from \cite{lishak} to this equality. We can view it as a van Kampen diagram with boundary $w_{n,1}$ over $\langle G| w_{n,1} \rangle$, because we have the equality in the free group: $w_{n,1} =  (v_2^{-1})^{-1} w_{n,1} (v_2^{-1}) \cdot (v^{-1}_1 k )^{-1}w^{-1}_{n,1}(v^{-1}_1 k ) \cdot (k)^{-1}  w_{n,1} (k)$. Since the boundary of this diagram is also $w_{n,1}$ it can be viewed as a spherical van Kampen diagram over $\langle G| w_{n,1} \rangle$ (see Figure \ref{ch5:fig:sphere1}), call it $D$. 
\par
We want to apply Theorem 3.10 from \cite{lishak} to $D$. This theorem says that if there is an HNN-extension ($G$ in this case) and some added relations ($w_{n,1}$ in this case) satisfying effective small cancellation conditions, then elements in $G$ which are not trivial in $G$ can not be effectively trivial in $\langle G| w_{n,1} \rangle$. In \cite{lishak} it was checked that the relator $t^{-1}w_{n,1}$ satisfies the effective metric condition $C'(\frac{1}{6})$ and therefore condition $(6,3)$. Similarly, $w_{n,1}$ satisfies $(6,3)$. We don't have the metric condition $C'(\frac{1}{6})$ (our pieces can be \emph{exactly} of length $\frac{1}{6}$), but it is not needed in the proof of that theorem. Secondly, our diagram is spherical, not planar, but we can just replace the Euler characteristic of the disk ($1$) with that of the sphere ($2$) in the proof of Theorem 3.10 from \cite{lishak} to get the same result. In the proof of that theorem the diagram is first made $w_{n,1}$-reduced, i.e. $w_{n,1}$-cells having large common boundaries are canceled out, then the contradiction is reached. We can not do that with spherical diagrams, because a spherical diagram might just disappear with the last cancellation not giving us a contradiction. Then Theorem 3.10 from \cite{lishak} implies that there is a cancellation in our diagram.

\par

Cancellations happen when there are a lot of $t$-bands between $w_{n,1}$-cells. In \cite{lishak} it was shown that in this case these bands all have length $0$, i.e. there is a pair of $w_{n,1}$-cells that touch each other. Consider two cases. Case 1: there is a cancellation involving the first $w_{n,1}^{-1}$ cell (conjugated by $v_2$) and case 2: a cancellation with the other $w_{n,1}^{-1}$ cell. In case 1 we have two options: $p v_2^{-1} k^{-1} v_1 p^{-1}$ (see Figure \ref{ch5:fig:sphere2} for the definition of $p$) is a loop on the diagram, or $v_2$ is. Note, loops in $D$ represent $1$ in $G$ because $\langle G| w_{n,1} \rangle = G$ as groups. By assumption, $v_2 \neq_G 1$, therefore $p v_2^{-1} k^{-1} v_1 p^{-1} =_G 1$, or $k v_2=_G v_1$. After canceling this pair of $w_n$-cells we obtain a spherical diagram $D'$ that has one pair of $w_n$-cells. In $D$ these cells were connected by a curve spelling $k$. After the cancellation the curve spells $k'=_G k$. Note, $k$ might not be equal to $k'$ effectively. For example they might differ by one of the removed cells. In $D'$ the two remaining cells have to cancel each other, therefore $k'$ is a loop and thus $k=_G k'=_G 1$. We know from before $k v_2=_G v_1$, and so $v_2=_G v_1$. Case 2 is dealt with similarly.

\end {proof}

\begin{figure}[h]
\centering
\includegraphics[scale=0.7]{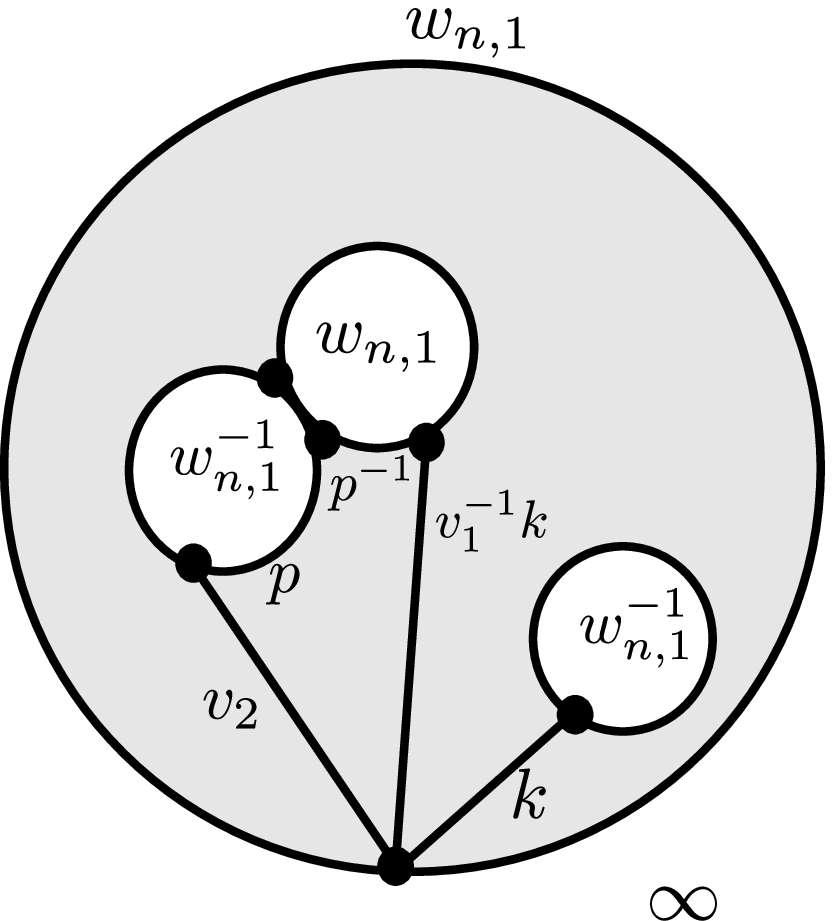}
\caption{Two cells are touching along a $t$-cable of $0$ thickness. Since this $t$-cable is long enough the cells touch in the same position.}
\label{ch5:fig:sphere2}
\end{figure}

By combining Lemma \ref{ch5:s2} and a result from \cite{bridson3} we can get the following lemma. We reproduce the proof here for completeness.

\begin {lem} \label{ch5:hat1}
Let $w$ be a non-empty freely reduced word in $\langle s,x \rangle$. We can think of $w$ as a word in $H_v$. Then if $w=_{H_v} 1$, $Area_{H_v}(w) \geq E_n$.
\end{lem}

\begin {proof}
Suppose the area is less than $E_n$. First we prove, that circular $s$-bands are impossible. Find an innermost circular $s$-band. Then we see that either a power of $t$ or a power of $v w_n^{-1} v^{-1} w_n$ is $1$ in $G$ through area $< E_n$. This is impossible by Lemma \ref{ch5:s1}. We checked the conditions required for an application of Lemma \ref{ch5:s1} in the proof of Lemma \ref{ch5:s2}. Semi-circular $s$-bands are impossible because $1 \neq_G x^i \neq_G t^j$. Therefore $w=_G 1$, which is impossible, a contradiction.
\end {proof}

\begin {lem} \label{ch5:hat2}
Let $r_1, r_2$ be non-empty freely reduced words in $\langle s,x \rangle$. If $ g r_1 g^{-1} =_{H_v} r_2$ (via area $<N < E_n$) for some $g \in H_v$, then either $g$ is equal to a word in $\langle s,x \rangle$ in $H_v$ via area $< N$ or $r_1 = x^i$, $r_2 = x^j$ for some $i,j$.
\end{lem}

\begin {proof}
Consider a diagram for $ g r_1 g^{-1} =_{H_v} r_2$. As in the proof of Lemma \ref{ch5:hat1} there are no circular $s$-bands, so the subdiagrams between $s$-bands are over $G$. Suppose there is a letter $s^{\pm 1}$ in $r_1$ (otherwise $r_1, r_2$ are powers of $x$). Then the $s$-band originating on it has to go to $r_2$ for the same reasons explained in the proof of Lemma \ref{ch5:hat1}. The subdiagrams between such $s$-bands are in $G$, and we can use Lemma \ref{ch5:s2} to show that these $s$-bands have length $0$. That shows $g$ is equal to a word in $\langle s,x \rangle$ in $H_v$ via area $< N$.
\end {proof}

\section{Effective Isomorphisms} \label{ch5:pseudo-groups}
The goal of this section is to prove that there are no ``effective'' homomorphisms (with certain ``effective injectivity'' conditions) between $\mu_v$ and $\mu_{v'}$ for $v \neq v'$. We are going to use the idea of Collins (\cite{collins}): Let $F:K \to K$ be a homomorphism. Since the Baumslag-Solitar group ($K$) is an HNN extension of $\langle x \rangle$ with the stable letter $y$, elements in $K$ have the invariant $y$-length (the number of $y$, $y^{-1}$ letters in a reduced form). Collins' lemma implies that conjugation can not change $y$-length, if the conjugated elements are cyclically $y$-reduced. Since we have a conjugation $(F(y))^{-1}F(x) F(y) = (F(x))^2$, we see that $y$-length of $F(x)$ has to be $0$, and that's the first step in describing $F$.

\par

Now, $G$ is an HNN-extension of $K$ and therefore, by the same logic, $t$-length of the image of $x$ is $0$. This was used in \cite{brunner} to classify homomorphisms $G \to G$. We are going to extend this logic to $H_v$, an ``effective'' HNN extension of $G$, and finally $\mu_v$ an ``effective'' amalgamated free product of $H_v$ with itself (in free products with amalgamation the invariant length is the number of factors in the reduced form). To do this we need to ``effectivise'' Collins' argument, which can be done by using van Kampen diagrams (see also a classification of maps $G \to G$ using diagrams in \cite{lishak2}). Then we will use our knowledge of the ``effective'' structure of $\mu_v$ to complete the proof. This effective structure is somewhat explicated by the technical lemmas of the previous section.

\begin{defn}
For a presentations $\mu$, denote by $[\mu]$ the free group on the letters of the presentation.
\end{defn}

\begin{defn} \label{effectivemap}
For presentations $\mu, \mu'$, we say a group map $F: [\mu] \to [\mu']$ is $(L,N)$ (we assume $N>L$) if for a letter $a \in \mu$,  $F(a)$ has length $< L$, and for a relator $u$ of $\mu$, $Area_{\mu'}(F(u)) < N$. If in addition $Area_{\mu'}(F(a)) > M$ (possibly $F(a) \neq_{\mu'} 1$), then we say $F$ is $(L,N,M)$.
\end{defn}

\begin{rem}
One should think of an (L,N) map $F$ as an ``effective'' group homomorphism. 
\end{rem}

Let $v_1, v_2$ be words like in Lemma \ref{ch5:solitar}. Let $\mu_1 = \mu_{v_1}, \mu_2 = \mu_{v_2}$, $l=\max\{l(\mu_1), l(\mu_2)\}$. In this section we will prove several lemmas about an effective homomorphism from $\mu_1$ to $\mu_2$. We start with one and then simplify it until in the end we are able to prove that it exists only if $v_1=v_2$. Each modification of the homomorphism either would be trivial on the level of pseudogroups, or a composition with a simple automorphism of $\mu_2$: a conjugation or the transposition of the two factors of $\mu_2$.

\par

The next lemma is the first instant of the ``effectivized'' version of the Collins' argument described in the beginning of this section. We also have to adapt this argument to free products with amalgamation instead of HNN extensions.

\begin {lem} \label{ch5:aut1}
If $F: [\mu_1] \to [\mu_2]$ is $(L,N,M)$ (where $lN^3<E_n$), then there exists $F': [\mu_1] \to [\mu_2]$ of type $(3N^3,lN^3, M-lN^3)$ such that $F'(x) \in \langle x,y,t,s \rangle$.
\end{lem}

\begin{proof}
Diagrams over any free product with amalgamation consist of subdiagrams of cells from one or the other factor (regions), bounded by the mixed cells (the cells responsible for amalgamation). In the case of $\mu_2$ we have regions made up from $H_{v_2}$ cells and regions from $\hat H_{v_2}$ bounded by the cells corresponding to $s=\hat x, x = \hat s$. Therefore the pure regions are bounded by words in $\langle x,s \rangle$ (or $\langle \hat x, \hat s \rangle$) and pieces of the boundary of the diagram. 

Consider a diagram for $F(x)^{F(y)} = (F(x))^2$ of area less than $N$. Our goal is to find $F'$ such that $\hat l(F'(x)) = 1$. Suppose $\hat l(F(x)) \geq 2$, then $\hat l(F^2(x)) \geq 4$. Consider the regions of the diagram in $H_{v_2}$. In total they have twice as many boundary pieces on $F^2(x)$ than on $F(x)$. Therefore there either exists a ``no-hat'' region with boundary on $F^2(x)$ only (case 1), or a region with more than one boundary piece on $F^2(x)$ (case 2), see Figure \ref{ch5:fig:freeproduct}. In case 1 we have the equality of a no-hat piece of the boundary of the diagram to a word in $\langle \hat x, \hat s \rangle$, in case 2 we have the equality of a mixed piece (the part of the boundary between some two no-hat pieces) to a word in $\langle x, s \rangle$. In both cases we can find a word $w =_{\mu_2} k^{-1}F(x)k$ such that $Area_{\mu_2}(w^{-1}k^{-1}F(x)k) <N$, the length of $k$ is less than $L$, the length of $w$ is less than $L+N$ ($N$ is a bound for the length of the words in $\langle \hat x, \hat s \rangle$ or $\langle x, s \rangle$ -- the boundaries between pure regions), and $\hat l(w) < \hat l(F(x))$. Thus, we define $F^{(1)}$ of type $(L+N+L, N+lN, M-N)$ to be $F^{(1)}(x) = w$, for $a \neq x$ define $F^{(1)}(a) = k^{-1}F(a)k$. We have $\hat l(F^{(1)}(x)) < \hat l(F(x))$.

We can continue in this way defining $F^{(2)}$, $F^{(3)}$, etc. Until $\hat l(F^{(m)}(x))=1$ for some $m \leq \hat l(F(x)) \leq L$. We can estimate $F^{(m)}$ to be $(L+N+4N+7N+10N+...+(3L+1)N = L+LN+3N(L+1)L/2 < 3NL^2, lNL^2)$ or $(3N^3,lN^3)$. Define $F'$ to be either $F^{(m)}$ or $\hat F \circ F^{(m)}$, where $\hat F$ is the automorphism on $[\mu_2]$ interchanging $H_{v_2}$ with $\hat H_{v_2}$.
\end{proof}

\begin{figure}[h]
\centering
\includegraphics[scale=0.7]{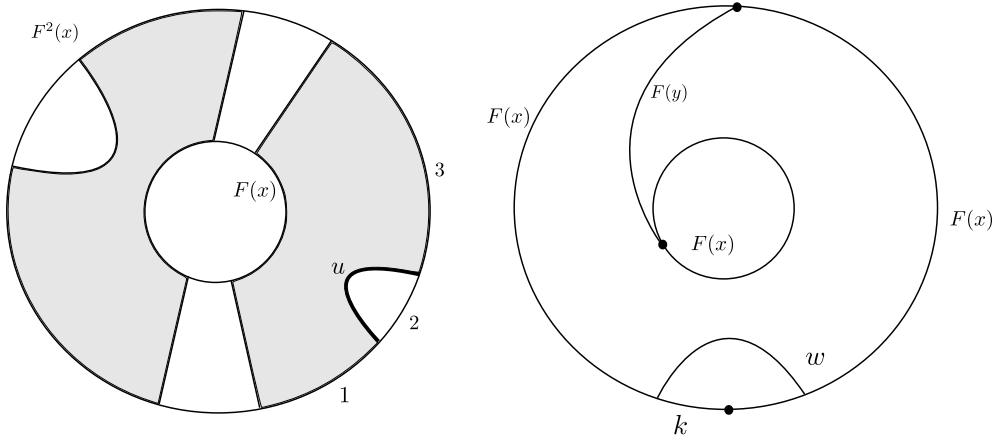}
\caption{Two annular van Kampen diagrams for $F(x)^{F(y)} = (F(x))^2$. The left diagram is an example of $\hat l(F(x)) = 4$. If the white regions are the no-hat regions, and the grey ones are hat regions,  then it represents case 1: the piece of the boundary marked by $2$ can be replaced by $u$, thus reducing $\hat l(F(x))$. If, on the other hand, the grey regions are no hat regions, then we can see case 2: one of the two grey regions has two pieces on the outer boundary (piece 1 and piece 2), they are connected by a piece $u$ that can be used to replace piece $2$. On the right diagram we see what happens if this reduction goes over the point of concatenation of two $F(x)$ -- we can conjugate by $k$ to recover $F(x)$ from $w$.}
\label{ch5:fig:freeproduct}
\end{figure}

\begin {lem} \label{ch5:aut2}
If $F: [\mu_1] \to [\mu_2]$ is $(L,N,M)$ (where $lN^3<E_n$) is such that $F(x)$ is a word in $H_{v_2}$, then there exists $F': [\mu_1] \to [\mu_2]$ such that $F'(x), F'(y)$ are in $H_v$, where $F'$ is $(N^3,lN^3, M-lN^3)$.
\end{lem}

\begin {proof}
Consider a diagram for $F(x)^{F(y)} = (F(x))^2$ of area less than $N$. Our goal is (as in the previous lemma) to decrease $\hat l(F(y))$ to $1$. Observe that by applying Lemma \ref{ch5:hat1} we see that there are no topologically trivial regions on the diagram. Therefore, this diagram looks like concentric circles of alternating type (hat, no-hat), see Figure \ref{ch5:fig:freeproduct2}. Since both components of the boundary are in $H_{v_2}$, if there is more than one circle (i.e. $\hat l(F(y)) > 1$), there has to be more than two ($\hat l(F(y)) > 2$). This gives us a subdiagram (the second annulus) to apply Lemma \ref{ch5:hat2}. Then if $g$ from the conclusion of the lemma is a word in $\langle s,x \rangle$ we can obtain $F^{(1)}$ having $\hat l(F^{(1)}(y)) = \hat l(F(y)) - 2$, or if the conclusion is $r_1 = r_2 = x^j$ we can define $F^{(1)}$ (by conjugating by $g_1 \in H_{v_2}$) such that $F^{(1)}(x) = x^j$. In both cases $F^{(1)}$ is $(L+N+L, N+lN, M-N)$.

Now we apply this process to $F^{(1)}$, but we notice that if $F^{(1)}(x) = x^j$ the conclusion of the Lemma \ref{ch5:hat2} can not be of the second case, because $x^j$ is not conjugate to $s^k$ in $H_{v_2}$ (the conjugation corresponding to the first annulus). Therefore, an application of Lemma \ref{ch5:hat2} will always decrease $\hat l(F(y))$ by $2$ except possibly once. Similarly to the previous lemma we obtain $F'$ of type $(N^3,lN^3, M-lN^3)$.
\end {proof}

\begin{figure}[h]
\centering
\includegraphics[scale=0.7]{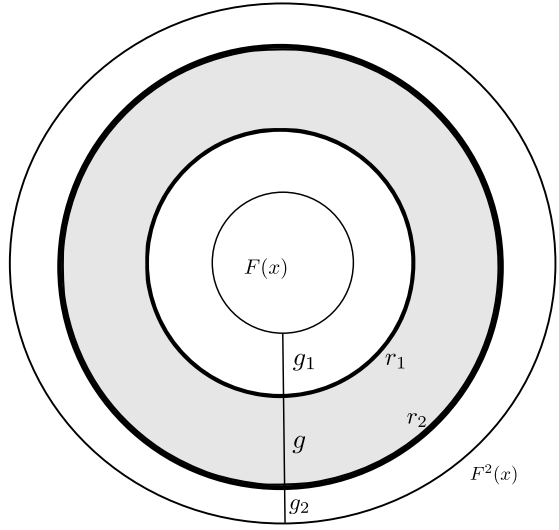}
\caption{A possible diagram fo $F(x)^{F(y)} = (F(x))^2$ with $\hat l(F(y))=3$. Here $g_1 g g_2 = F(y)$. We apply Lemma \ref{ch5:hat2} to $g r_2 g^{-1} = r_1$.}
\label{ch5:fig:freeproduct2}
\end{figure}

\begin {lem} \label{ch5:aut3}
If $F: [\mu_1] \to [\mu_2]$ is $(L,N,M)$ ($lN^3<E_n$) is such that $F(x), F(y)$ are words in $H_{v_2}$, then there exists $F': [\mu_1] \to [\mu_2]$ such that $F'(x), F'(y),F'(t)$ are in $H_{v_2}$, where $F'$ is $(N^3,lN^3, M-lN^3)$. 
\end{lem}

\begin {proof}
The proof is completely analogous to the proof of Lemma \ref{ch5:aut2}, the only difference is we consider the diagram for $F(x)^{F(t)} = F(y)$. Note that the only thing we used about a diagram for $F(x)^{F(y)} = (F(x))^2$ in Lemma \ref{ch5:aut2} is that its boundary is in $H_{v_2}$. As in the proof of Lemma \ref{ch5:aut2} we might need to conjugate $F$ to deal with the possible conclusion of Lemma \ref{ch5:hat2} not resulting in the decrease of $\hat l(F(t))$, but we conjugate by something in $H_{v_2}$ thus not interfering with $F(x), F(y)$ being in $H_{v_2}$.

\end {proof}

\begin {lem} \label{ch5:aut4}
If $F: [\mu_1] \to [\mu_2]$ is $(L,N,M)$ ($lN^3<E_n$) is such that $F(x), F(y), F(t)$ are words in $H_{v_2}$, then there exists $F': [\mu_1] \to [\mu_2]$ such that $F'(x), F'(y), F'(t),F'(s)$ are in $H_{v_2}$, where $F'$ is $(N^3,lN^3, M-lN^3)$. 
\end{lem}

\begin {proof}
We proceed as in Lemma \ref{ch5:aut3} by considering a diagram for $(v w_n^{-1} v^{-1} w_n) ^{F(s)} = F(t)$.
\end {proof}

\begin {lem} \label{ch5:aut5}
If $F: [\mu_1] \to [\mu_2]$ is $(L,N,M)$ ($lN^3<E_n$) is such that $F(x), F(y), F(t), F(s)$ are words in $H_{v_2}$, then there exists $F': [\mu_1] \to [\mu_2]$ of type $(lN^3, lN^3,M-lN^3)$ such that $F'(x)$ is in $G$, while $F'(y), F'(t), F'(s)$ are in $H_{v_2}$.
\end{lem}

\begin {proof}
We mimic Lemma \ref{ch5:aut1}, with $s$-bands instead of regions. Consider a diagram for $F(x)^{F(y)} = (F(x))^2$ of area less than $N$. If $l_s(F(x))>0$ then there is an $s$-band attached to the $(F(x))^2$ part of the boundary. One difference is that the length of $F(x)$ can grow faster -- by the length of $v w_n^{-1} v^{-1} w_n$ which is less than $l$.

\end {proof}

\begin {lem} \label{ch5:aut6}
If $F: [\mu_1] \to [\mu_2]$ is $(L,N, M)$ ($N < E_{n-1}$) is such that $F(y), F(t), F(s)$ are words in $H_{v_2}$ and $F(x)$ is in $G$, then $F(y) \in G$.
\end{lem}

\begin {proof}
Consider the diagram for $F(x)^{F(y)} = (F(x))^2$. Suppose $l_s(F(y))>0$, then by applying Lemma \ref{ch5:s4} to the concentric subdiagrams between $s$-bands we see that $F(x)$ is conjugate to $A^i$ while $F^2(x)$ is conjugate to $B^{i}$ (Lemma \ref{ch5:s4} makes sure it's the same $i$), where $A,B$ are either $t$ or $(v w_n^{-1} v^{-1} w_n)$. Then we have that $A^{2i}$ is conjugate to $B^i$, so that $A=B$ and again by Lemma \ref{ch5:s4} that's impossible ($i \neq 0$ because $M>N$).

\end {proof}

\begin {lem} \label{ch5:aut7}
If $F: [\mu_1] \to [\mu_2]$ is $(L,N,M)$ ($lN^3<E_n$) is such that $F(t), F(s)$ are words in $H_{v_2}$, $F(x), F(y)$ are in $G$ and $Area(F(x))>N$, then there exists $F': [\mu_1] \to [\mu_2]$ such that $F'(x)=x^i$ and $F'(y) \in \langle x,y \rangle$, where $i\neq 0$, while $F'(t), F'(s)$ stay in $H_{v_2}$, $F'$ is $(3N^3, lN^3,M-lN^3)$.
\end{lem}

\begin {proof}
We start similarly to Lemma \ref{ch5:aut1}, by making sure $F(x)$ has no $t$ letters, then no $y$ letters ($i \neq 0$ follows from $M>lN^3$).  Then, similar to Lemma \ref{ch5:aut2}, we are going to be reducing $F(y)$ until its $l_t$ is equal to $0$. See Figure \ref{baumslagaut} for a diagram.

\end {proof}

\begin{figure}[h]
\centering
\includegraphics[scale=0.7]{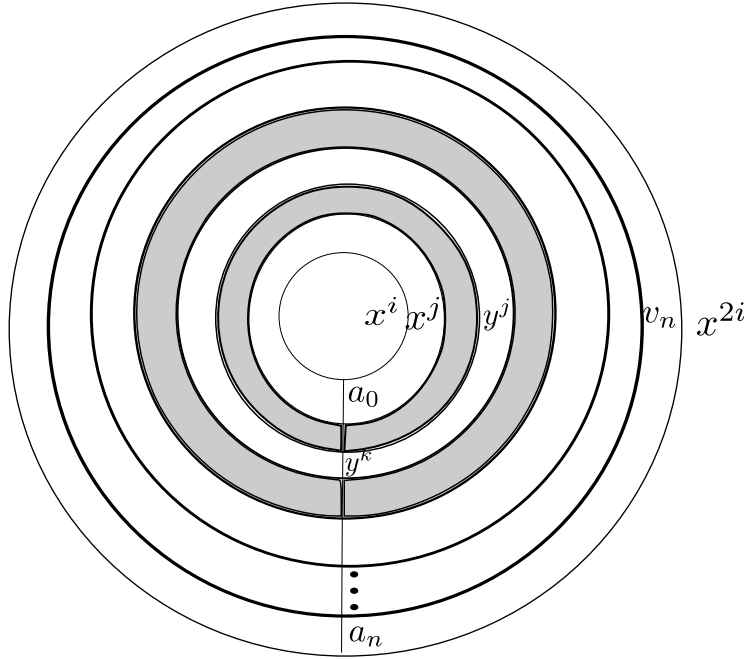}
\caption{The grey circular regions are $t$-bands. Since the inside of the innermost $t$-band is $x^j$, the inside of the next $t$-band is $y^j$, which implies there is a pinch in $F(y)$ that we can remove.}
\label{baumslagaut}
\end{figure}

\begin {lem} \label{ch5:aut8}
If $F: [\mu_1] \to [\mu_2]$ is $(L,N,M)$ ($N < E_{n}$) is such that $F(t), F(s)$ are words in $H_{v_2}$, and $F(x)=x^i, F(y) \in K$, then $F(t) \in G$.
\end{lem}

\begin {proof}
Consider the diagram for $F(x)^{F(t)} = F(y)$. If there is an $s$-band it means $x^i$ is conjugate to $t^j$ or $1$ in $G$, which is impossible since $i \neq 0$ (because $M>0$).
\end {proof}

\begin {lem} \label{ch5:aut9}
If $F: [\mu_1] \to [\mu_2]$ ($2lN^3<E_n$) is $(L,N,M)$ is such that $F(s)$ is a word in $H_{v_2}$, $F(t) \in G$, $F(x)=x^i$ and $F(y) \in K$, then there exists $F': [\mu_1] \to [\mu_2]$ such that $F'(x) = x, F'(y) = y, F'(t) = gt, F'(s) \in H_{v_2}$, where $g$ is in $K$ and commutes with $x$, $F'$ is $(6N^3, 2lN^3,M-2lN^3)$.
\end{lem}

\begin {proof}
This is an effective version of the Lemma 3.4 from \cite{lishak2}. First we can reduce $F(t)$ until it has just $1$ letter $t$, similarly to Lemma \ref{ch5:aut7}. Then consider an equality in $K$ $(x^i)^{F(y)} = x^{2i}$ . Note that the total $y$ power of $F(y)$ (counted with signs) is $1$. Since $F(x)$ and $F(y)$ are conjugate in the Baumslag-Gersten group (by $F(t)$), $F(x)$ is conjugate to $x$. If in addition we conjugate $F$ by $yF(y)$, we get $F'(y)=y$. Let $F'(t) = g_1tg_2$. We see that $g_1=_K y^m$ and therefore we can make $F'(t) = gt$, where $g$ commutes with $x$.

\end {proof}

\begin {lem} \label{ch5:aut10}
If $F: [\mu_1] \to [\mu_2]$ of type $(L,N,M)$ is such that $F(s)$ is a word in $H_{v_2}$, $F(x)=x, F(y)=y$ and $F(t) = gt$, where $g \in K$ commutes with $x$, then there exists $F': [\mu_1] \to [\mu_2]$ such that $F'(x) = x, F'(y) = y, F'(t) = t$, and $F'(s) = hsk$ for $h,k \in G$, where $F'$ is $(N^3,lN^3,M-lN^3)$.
\end{lem}

\begin {proof}
Consider a diagram for $F(v_1^{-1}w_n v_1 )^{F(s)} = F(t)$ of area less than $N$. In $G$ $F(v_1^{-1}w_n v_1 )$ can not be conjugate to $F(t)$, because of different total power of $t$ (counted with signs). Therefore $F(s)$ has at least one letter $s$ or $s^{-1}$. Consider the nearest to $F(t)$ circular $s$-band on the diagram. Since $F(t)=gt$ is not conjugate to $1$ or $t^i$ for $i \neq 1$, we have that this $s$ band has length $1$, and $gt$ is conjugate to $t$ in $G$. Consider the subdiagram corresponding to this conjugation. There is one $t$ band on that subdiagram, giving $x^i g y^i=1$ in $G$, which is possible only for $i$ being $0$ and $g$ being $1$, because $g$ commuting with $x$ means it has the total $y$ power (counted with signs) $0$. Now, we want to reduce $l_s((F(s))$ to $1$. Notice that by Lemma \ref{ch5:s4} whenever two $s$-bands face each other with their $t$ side, we can cancel them. Furthermore, it is not possible to have just two $s$-bands facing each other with their $1$ side because we know $F(t) \neq_G 1$. Now, the lemma follows by the standard argument.
\end {proof}

\begin {prop} \label{ch5:effective}
If $F: \mu_1 \to \mu_2$ of type $(N,N,M)$ is such that $M > 2(lN)^{21}$ and $2(lN)^{21} < E_{n-1}$, then $v_1 = v_2$.
\end {prop}

\begin {proof}
Apply Lemmas \ref{ch5:aut1}, \ref{ch5:aut2}, \ref{ch5:aut3}, \ref{ch5:aut4}, \ref{ch5:aut5}, \ref{ch5:aut6}, \ref{ch5:aut7}, \ref{ch5:aut8}, \ref{ch5:aut9}, \ref{ch5:aut10} successively to obtain $F'$ of type $(2(lN)^{7 \cdot 3},2(lN)^{7 \cdot 3}, M - 2(lN)^{7 \cdot 3})$ such that $F'(x) = x, F'(y) = y, F'(t) = t$, and $F'(s) = hsk$ for $h,k \in G$.

Consider a diagram for $F'(v_1^{-1}w_n v_1 )^{F'(s)} = F'(t)$ of area less than $E_n$, or equivalently $(v_1^{-1}w_n v_1)^{hsk} = t$. The diagram has an $s$-band of length $1$, from which we see that $k^{-1} (v_1 w_n^{-1} v_1^{-1} w_n) k = v_2 w_n^{-1} v_2^{-1} w_n$ through area $< E_n$, which is impossible by Lemma \ref{ch5:s5} unless $v_1=v_2$. 

\end {proof}

\section{Exponential Growth} \label{ch5:main}
We will use the following Tietze transformations. Note the length bound in $Op_5(k)$.

\begin{defn}
 Let $\mu = \langle x_1,...,x_r \vert a_1,..., a_p\rangle$.

Elementary Tietze transformations: 
\begin{description}

\item[${Op}_1$] $\mu$ is replaced by $\langle
 x_1,...,x_r   |a_1,...,a_{i-1}, a' x_j^{\epsilon} x_j^{-\epsilon} a'', a_{i+1}, ..., a_p \rangle $, where $a_i \equiv a'a''$ and $\epsilon = \pm 1$.

\item[${Op}_1^{-1}$] The inverse of ${Op}_1$ - it deletes $ x_j^{\epsilon} x_j^{-\epsilon}$in one of the relators.

\item[${Op}_2$] $\mu$ is replaced by $\langle x_1,...,x_r\ | a_1,...,a_{i-1}, a'_i, a_{i+1}, ..., a_p \rangle $, where the word $a'_i$ is a cyclic permutation of the word $a_i$.

\item[${Op}_3$] $\mu$ is replaced by $\langle x_1,...,x_r | a_1,...,a_{i-1}, a_i^{-1}, a_{i+1}, ..., a_p \rangle$.

\item[${Op}_4$] $\mu$ is replaced by $\langle x_1,...,x_r | a_1,...,a_{i-1}, a_i a_j, a_{i+1}, ..., a_p \rangle$, where $i \neq j$.

\item[${Op}_5(d)$] $\mu$ is replaced by the presentation $\langle x_1,...,x_r, x_{r+1} | a_1, ..., a_p, x_{r+1}w \rangle$, where $w$ is a word on $a_1, ..., a_p$ of length at most $d-1$.

\item[${Op}_5^{-1}(d)$] The inverse of ${Op}_5(d)$.

\item[${Op}_6$] $\mu$ is replaced by the presentation $\langle x_1,...,x_r | a_1, ..., a_p, * \rangle$, where $*$ denotes the empty relation.

\item[${Op}_6^{-1}$] The inverse of ${Op}_6$.

\end{description}
\end{defn}

For any $d$, all presentations of the trivial group can be transformed from one to another using these operations. We imposed the bound $d$ so that Tietze transformations give rise to \emph{effective} homomorphisms.

\begin{defn}
For presentations $\mu, \mu'$, denote by $T_d(\mu, \mu')$ the minimal number of Tietze transformations needed to go from one to another.
\end{defn}

We describe a connection between the number of Tietze trasformations and effective isomorphisms in the next proposition.

\begin {prop} \label{ch5:tietzek}
If for some presentations $\mu, \mu'$ of the trivial group we have $T_d(\mu, \mu') \leq N$, then there exists $F: [\mu] \to [\mu']$ of type $(d^N, 1+2^N)$. Furthermore, if the letters in $\mu$ have area greater than $M$ then $F$ is $(d^N, 1+2^N, 2^{-N}M)$.
\end {prop}

\begin {proof}
Suppose $\mu$ differs from $\mu'$ by one operation. There is the obvious map $F: [\mu] \to [\mu']$ (if a letter was removed, send it to $w$, the word from the definition of $Op_5(d)$). In the case of ${Op}_1^{ \pm 1}, {Op}_2, {Op}_3, {Op}_5^{ \pm 1}(d), {Op}_6^{ \pm 1}$, $F$ is $(d,2)$. In the case of ${Op}_4$, $F$ is $(2,3)$. Therefore, if $T_d(\mu, \mu') < N$ then $F$ is $(d^N, 1+2^N)$. Similarly, if we have a bound for the area of a generator $M$, then after one transformation it might decrease by at most $2$, so that it becomes $\frac{M}{2}$. Therefore $F$ is of type $(d^N, 1+2^N, 2^{-N}M)$.
\end {proof}

\begin{defn}
If $S$ is a set of presentations of the trivial group of length at most $l$ and of rank $k$ such that for any $\mu \neq \mu'$ in $S$, $T_d(\mu, \mu') > exp_m(l)$, we call $S$ $(l,m,k)_d$-disconnected. Let $M(l,m,k)_d$ denote the maximal size of $(l,m,k)_d$-disconnected sets.
\end{defn}

 We are going to show that for any $m$, $M(l,m,k)_d$ grows at least exponentially for $d\leq l$.

\begin {thm} \label{ch5:mainthmk}
For each $m$ there exists a constant $const(m)$ such that if $l>const(m)$ and $d\leq l$, then $M(l,m,8)_d > {1.18}^l$.
\end {thm}

\begin {proof}
Choose the smallest $n$ such that $E_{n-1}  > d^{\exp_{m}(22l) }$. This inequality is implied by $E_{n-3}  > 2 \log_2 (22 \exp_{m}(l)) $, therefore $n$ is less than $log_2(l)-20$ for large enough $l$.

\par

Let $S_{n,l} = \{ \mu_{v} | l(\mu_{v}) \leq l \}$, where $v$ is any word of the form described in Lemma \ref{ch5:solitar}. Note that $n$ was already used in the definition of $\mu_v$.
Suppose $T_d(\mu,\mu') \leq exp_m(l)$, then by Proposition \ref{ch5:tietzek} there exits $F: [\mu] \to [\mu']$ of type $(d^{\exp_m(l)}, 1+ exp_{m+1}(l), \frac{E_n}{exp_{m+1}(l)} )$ (we used here the fact that the identity map on $[\mu]$ is $(2,2,E_n)$). To apply Proposition \ref{ch5:effective} to $F$ we need to check that $2(l (d^{\exp_{m}(l))} )^{21} < E_{n-1} $ and $\frac{E_n}{exp_{m+1}(l)} > 2(l (d^{\exp_{m}(l))} )^{21}$. The first inequality follows because $2(l (d^{\exp_{m}(l))} )^{21}  < d^{\exp_{m}(22l) } < E_{n-1}$ by the choice of $n$ (for large enough $l$), the second one holds because $E_n > (E_{n-1})^2$. Therefore $\mu = \mu'$ and $S_{n,l}$ is $(l,m,8)_d$-disconnected for large enough $l$.

\par

Now we want to estimate the size of $S_{n,l}$. Recall that $l(\mu_v)<2l(v)+200 \cdot 2^n+20$. Using our estimate for $n$ we obtain $l(\mu_v)<2l(v)+0.01 l$. Therefore to make sure $\mu_v \in S_{n,l}$, it is enough to make $l(v) < \frac {0.99l}{2}$. There are at least $2^{0.99l/4} > 1.18^{l}$ such $v$.
\end {proof}

We use this theorem in the following application, for which we need a definition.

\begin{defn}

Let $\Gamma_{l,m,k}$ be a graph, the vertices of which are the balanced presentations of the trivial group of rank $k$ and of length at most $l$. Any two presentations are connected by an edge if they require at most $exp_m(l)$ Tietze transformations (for $d=l$) to go from one to another.
\end{defn}

\begin {thm}
For each $m$, there exists a constant $const(m)$ such that if $l>const(m)$, then the number of connected components of $\Gamma_{l,m,8}$ is at least ${1.18}^l$.
\end {thm}

\begin {proof}
We can make a crude estimate that the number of vertices of $\Gamma_{l,m,8}$ is less than $8^l$. Therefore, if two vertices $\mu, \mu'$ are connected by some path, then $T_{l}(\mu, \mu') < 8^l exp_m(l)$. Noting that $exp_{m+1}(l) > 8^l exp_{m}(l)$ for large enough $l$, we see that this theorem follows from $M(l,m+1,8)_{l}> {1.18}^l$ (Theorem \ref{ch5:mainthmk}).

\end {proof}

This theorem will give rise to exponentially many triangulations of the standard $4$-sphere. In order to obtain triangulations of an arbitrary compact manifold (with fundamental group $P$), we will need the following proposition.

\begin {thm} \label{nontrivial}
Let $P$ be a finite presentation. For every $m$, if we choose $l$ large enough, the ${\lfloor 1.18 \rfloor }^l$ presentations $\mu_v$ from the previous theorem have the property that $P * \mu_v$ are pairwise $exp_m(l)$ distant.
\end {thm}

\begin {proof}

From Proposition \ref{ch5:tietzek}, we obtain $F: [P * \mu] \to [P * \mu']$ of type $(d^N, 1+2^N)$, where $N$ is the distance. We can not claim that $F$ is $(d^N, 1+2^N, 2^{-N}M)$ anymore, because a generator from $P$ might have a small area; but for the generators of $\mu$, the area of their images under $F$ is still large. Consider a van Kampen diagram for $F(x)$. Since there is no amalgamation in the product $P * \mu'$, the diagram for $F(x)$ is a bouquet of diagrams over the individual factors. We would like to define $F'$ by contracting those regions of the diagram corresponding to the factor $P$. Unfortunately, since the Dehn function of $P$ could be rapidly growing, we can not guarantee that those regions are going to be smaller than the regions of $\mu'$ even for very large $l$, thus $F'$ might not be of type $(d^N, 1+2^N)$ any longer.

\par

We can say $F': [P * \mu] \to [P * \mu']$ is of type $(d^N, 1+2^N, 2^{-N}M)_2$ (``$2$'' stands for the second factor), where we count only cells coming from $\mu'$ and consider only the images of the generators of $\mu$ for the $2^{-N}M$ estimate. Then we can go with all the proofs of the previous section with $(\ ,\ ,\ )$ replaced by $(\ ,\ ,\ )_2$ and with the proof of Theorem \ref{ch5:mainthmk}, thus concluding that $N$ has to be large.

\end {proof}

For one of the applications in the last section it is not convenient to use the distance between presentations defined above, but instead we will use the notion of {\it effective
isomorphism}.  

\begin{defn} \label{effectiveiso}
Two presentations $\mu_1, \mu_2$ are $(L,N)$-effectively isomorphic, if there exists a pair of $(L, N)$-maps $F:[\mu_1] \longrightarrow [\mu_2]$ and $G:[\mu_2] \longrightarrow [\mu_1]$ such that for each generator $a \in \mu_1$,  $(G \circ F)(a)a^{-1}$ has area less than $N$ in $\mu_1$ and for each generator $b \in \mu_2$,  $(F \circ G)(b)b^{-1}$ has area less than $N$ in $\mu_2$. Such maps are called $(L,N)$-effective isomorphisms.
\end{defn}

\begin{rem} \label{effectivetietze}

The classical proof of the fact
that a pair of isomorphic finite presentations can be transformed one into the other by a finite sequence of Tietze transformations (cf. \cite{bhp} or Lemma 4.12 in \cite{holt} for a readily available text) leads to an explicit upper bound for the number of Tietze transformations required to connect two $(L, N)$ -effectively isomorphic finitely presented groups. This number is either polynomial or exponential in $L_1+L_2+$ the sum of the lengths of the finite presentations of $\mu_1$ and $\mu_2$ depending on a particular version of the definition of Tietze transformations. Therefore as a corollary of Theorem \ref{ch5:mainthmk} we get exponentially many presentations which are not $(exp_m(l),exp_m(l))$-effectively isomorphic.

\end{rem}

\section{Superexponential Growth} \label{superexp}

\begin{lem}
 Let $P=\langle f_1\ldots f_m\vert r_1,\ldots r_k \rangle$ be a finite presentation of a group
$G$. Then there exists another finite presentation $P'=\langle f'_1,\ldots, f'_M
\vert r'_1\ldots r'_K \rangle$ of the same group (for some $M$ and $K$) with the following properties:

\begin {enumerate}
\item Each generator $f'_i$ appears in at most $3k$ words $r'_j$;
\item Each relator $r'_j$ has length two or three;
\item $l(P')\leq const \ l(P)$, where $const$ is an absolute constant;
\item $K-M=k-m.$ In particular, if $P$ is balanced then $P'$ is balanced.
\item $P'$ is constructed from $P$ by means of an explicit algorithm. The number
of Tietze transormations (for $d=\log l(P)$) required to transform $P$ into $P'$ is bounded by
$const \ l(P)$, where $const$ is an absolute constant.
\item Consider multigraph $G(P')$ such that set of its vertices coincides with the set of generators of $P'$,
and two vertices are connected by an edge if and only if the corresponding generators or their inverses 
both appear in a
relator. If a relator has a form $a^2b$ or $aba$, then the graph contains a loop based at $a$ and two copies of the edge $ab$.
Also, one has a copy of an edge for each occurence of the corresponding pair of generators in a relator.
Then the diameter of each component of this graph does not exceed $const\ m\ln l(P)$. 
\end {enumerate}

Alternatively, one can remove the last condition but ensure that $f'_i$ appears in at most $3$ 
words $r'_j$.

\end{lem}

\begin{proof}
On the first stage we are going to rewrite each relation $r_i=e$ of length greater than three, where
$r_i=f_{i_1}^{\epsilon_1}
f_{i_2}^{\epsilon_2}\ldots f_{i_{l(r_i)}}^{\epsilon_{l(r_i)}}$ with $\epsilon_i=1$ or $-1$
as a sequence of $\sim l(r_i)$ new relations involving $\sim l(r_i)$ new generators.
We divide $r_i$ into consecutive words of length $2$ (followed by a single letter if $l(r_i)$ is odd.
We introduce a new generator and the corresponding relation for each word of length two that
appears in one of relators $r_i$. We plug new generators into words $r_i$ reducing their length
by approximately a factor of two. Then we repeat this process until after some step all relations
will have length at most three. The number of the required steps will be bounded by $O(\ln \max_i l(r_i))$.
Denote the number of steps required to turn $r_i$ into a relator of length $\leq 3$ by $s_i$, where
$s_i=0$ is $l(r_i)\leq 3$. Denote the generators added
at step $j$ by $f_{lj}$.
\par
The resulting finite presentation has property 2). In order to ensure
property 1) we are going to perform a second stage of modification of the finite
presentation. This stage would consist of $k$ independent steps, the $i$th step corresponding to
the relator $r_i$ in $P$. One by one consider all old or new generators $f=f_l$ or $f_{lj}$ that 
enter $r_i$. 
Consider all occurences of $f$ in the set of relators that includes the new version of $r_i$ as well as all relators
obtained on the first $s_i$ steps of the previous stage while considering $r_i$. Assume
that the number of occurences of $f$ in these relators is $T$. If $T\leq 3$, we do not do
anything. If $T>3$, then we introduce $T-3$ new generators
$w_1,\ldots, w_{T-3}$
together with new relations $f=w_1, w_1=w_2,\ldots , w_{T-4}=w_{T-3}$. Then
we replace all but four occurrences of $f$ in (old) relations
by different $w_i$, $i\in\{1,\ldots, T-4\}$ and another two occurrences by $w_{T-3}$. It is easy to see
that the new set of generators and relators satisfies property 1. 
\par
Alternatively, we can perform this stage not in $i$ ``disjoint'' steps corresponding to different
$r_i$, but immediately for all occurences of $f$ into all (rewritten) initial relators $r_i$ as well as relators
added on the first stage. If we do that, then each generator of the resulting presentation will appear in at most three relators.
\par
In both cases properties 3, 4 and 5 can be immediately seen for our construction. Yet only the first
version of the construction on stage $2$ yields the last property. Indeed, consider 
a component of $G(P')$. Take one of its vertices that corresponds to an original generator
$f_l$ of $P$. (It is easy to see that such a vertex exists.) Let $I_l$ 
denote the set of $i$ such that $f_l$
appears in $r_i$. For each $i\in I$ consider the induced subgraph $G_i$ of $G(P')$ such that its
set of vertices coincides with the set of generators corresponding to $r_i$ obtained by the end of stage 2
of the construction of $P'$. One of these vertices corresponds to $f_l$. The diameter of this graph
will be bounded by $O(\ln l(r_i))$, as each generator appears in a path of length $\sim \log_2 l(r_i)$
that corresponds to its appearnces in new generators that are being introduced at each of $s_i$
steps of the construction. Each pair of these paths either merge or can be connected by an edge that comes 
from $r_i$ being considered as a new relator of length $\leq 3$ after $s_i$th step of stage 1.
Therefore, the diameter of $G_i$ is also $O(\ln l(r_i))$. Consider $G^l=\bigcup_{i\in I_l} G_i$.
This is the union of all $G_i$ that intersect at the vertex corresponding to $f_l$; its diameter
is at most $2\max_{i\in I_l}diam G_i$. Now its is easy to see that the considered component
is the union of $G^l$ over all $l$ in a subset of $\{1, \ldots , m\}$, which immediately
implies the last property.
\end{proof}

\begin{defn}
Assume that a pair of finite presentations $P$ and $P'$ of the same
group satisfy properties 3, 4 and 5 in the previous lemma. In this case we say that $P'$ is an effective rewriting of $P$.
\end{defn}

\begin{defn}
If a finite presentation of a group satisfies properties
1, 2 and 6 from the previous lemma we say that it is nice.
\end{defn}

So in our new notations the previous lemma asserts that each finite presentation of a group
can be effectively rewritten as a nice presentation.

In the next lemma $N$ denotes a potentially large variable parameter and $C_1, C_2, C_3$ should be interpreted as constant parameters
that do not increase as $N \to \infty$.

\begin{lem}
Let $P$ be a presentation of a group $G$ with $C_1$ generators and $C_2$ relators of length $\leq C_3N\ln N$. There exists
an effective rewriting $P_0$ of $P$ such that $L(P_0)\leq C_4N$, where $C_4=C_4(C_1,C_2,C_3)$
does not depend on $N$. 
\end{lem}

\begin{proof}
Introduce $O({N\over \ln N})$ new generators as one-letter abbreviations 
of all words in the original set of generators up to the length
 $[\log_2 {N\over \ln N}]=[\log_2e (\ln N-\ln \ln N)]$.
The total length of these new relations is bounded by $O(N)$.
Now all words of length $\leq C_3N\ln N$ in old generators can be (effectively) rewritten as words
 of length $\leq C_5\ N$ in new
generators. We would like to do this for each of the old relators. 
The resulting finite presentation will be $P'$.
\end{proof}

The ``savings'' in length came from the fact that $O({N\over \ln N})$ words of length 
$\lfloor \log_2{N\over\ln N}\rfloor$ are repeated on the average $O(\ln N)$ times in any
word of length $O(N\ln N)$, when we consider it as a sequence of $O(N)$ words of length
$\lfloor \log_2{N\over\ln N}\rfloor$ followed by a shorter word at the end. Therefore it saves space
to abbreviate these short words (by introducing new generators).  

For a given $N$, let $n=\lfloor N\ln N\rfloor$. Recall that Lemma 6.1 has two versions: one where $P'$
satisfies conditions 1-6, and another where it satisfies conditions 1-5, but each generator
enters at most $3$ relations.
Applying Lemma 6.4 to each of the
 $\lfloor 1.18^n\rfloor$ balanced finite presentations that we constructed above and then applying
the second version of Lemma 6.1 we get the Theorem C. If we use Lemma 6.4 and then proceed
as in the first version
 of Lemma 6.1 we will obtain the following modification of Theorem C that will be used to prove Theorem A
in the next section.

\begin{thm} \label{superexppr}
There exists $C>0$ and for each $m$ there exists $C(m)$
such that for each $N>C(m)$ there exists at least $\lfloor N^{CN}\rfloor$ balanced finite presentations
of the trivial group of length $\leq N$ such that the minimal number of the Tietze moves required to transform one of these presentations into
another is greater than $\exp_m(N)$. In particular, the graph
$\Gamma_{N,m,N}$ has at least $N^{CN}$ components. Further, one can ensure that in each
of these presentations $P'$ each relation has
length at most three, each generator
enters at most $12$ relations, and that the graph $G(P')$ introduced in Lemma 6.1(6) has diameter
not exceeding $const \ln N$ for some absolute constant $const$.
\end{thm}

\begin{proof} Let $P$ be one of the finite presentations in Theorem B applied for $\lfloor N\ln N\rfloor$
instead of $N$.
Note that after the application of Lemma 6.4 the number of generators and relators of $P_0$ will become 
$O({N\over \ln N})$. Let $P_1$ is a finite presentation obtained from $P_0$ by applying Lemma 6.1.
Then the
diameter of the graph $G(P_1)$ will be bounded by $O(N)$. Further, some generators will enter
$O({N\over\ln N})$ relations. Fortunately, we can proceed somewhat differently. First, 
consider the process described in the proof of Lemma 6.4. We apply it
to each of four relators $r_i$ of $P$ separately. (As the result, we are going
to get four different generators denoting each ``short'' word
of length $\lfloor \log_2{N\over\ln N}\rfloor$.) 
We can plug this modification of $P$ in the process of modification introduced in the proof
of Lemma 6.1, and consider it
as stage 0 of the new modifed process. On stage 1 we consider
all original relations, as well as the relations added at stage 0 and break them into short
relations of length at most $3$. Now on the second stage we use four relators $r_i$ of 
the original presentation $P$. All new generators code subwords of exactly one of these four
relators. The original generators of $P$ can still enter $3\times 4=12$ relators ($3$ relators
for each of the original four relators). Stage 0 of the process can add a summand of at most
$O(\ln\ln{N\over \ln N})$ to the diameter, and so does not affect the logarithmic bound.
This helps to ensure that the diameter of $G(P')$ will be $O(\ln N)$.

\end{proof}

Observe that graphs $G(P')$ corresponding to the finite presentation of the trivial group that
were constructed in the proof of previous theorem
are connected graphs of degree $\leq 24$ with diameter $\leq Const \ln N$.
For geometric applications in the next section we need the following theorem:

\begin{thm} \label{graphthm}
Let $G$ be a connected multigraph with $N$ vertices of bounded degree linearly embedded in the PL-sphere
 $S^p=\partial\Delta^{p+1}$
of dimension $p>3$. Consider $p$ and the upper bound for the degree of vertices of $G$
as constants, and $N$ as large variable. Assume, further, that the diameter of $G$ does
not exceed $O(\ln N)$. Then there exists a triangulation of $S^p$ with $O(N\ln N)$ simplices
such that all vertices of $G$ and its edges are $0$-dimensional and $1$-dimensional simplices
of the triangulation.  
\end{thm}
\begin{proof}
First, we fill cycles corresponding to loops, then digons corresponding to multiple edges, so that
it remains to fill cycles of the corresponding graph $G'$ (where all loops are removed, and multiple edges are replaced 
by single edges).
Consider a cycle basis of $G'$ obtained from a spanning rooted tree. Each element of this basis
corresponds to an edge of $G'$ not in the spanning tree. This edge together with two
paths from the root to the endpoints forms a generator of the (free) fundamental group of $G'$;
this generator can be killed by attaching a $2$-cell to the boundary of a simple cycle
formed by the edge and two segments of the paths from the endpoints to the root that go to the
nearest common vertex $v$ of these paths. The length $L$ of this cycle does not exceed $2diam (G')+1=O(\ln N)$.
The $2$-cell can be subdivided into $L-2$ $2$-simplices with addition of $L-3$ new $1$-simplices
connecting $v$ with the vertices of the simple cycle.
We can perform the same operation to all simple cycles of the considered cycle basis, 
obtaining a simply connected simplicial $2$-complex $K$ with trivial reduced homology groups, and, thus, contractible .
The number of cycles does not exceed the number of edges of $G'$ which is $O(N)$. Therefore, the total number of simplices in $K$ will
be $O(N\ln N)$. 

Now consider a small simplicial tubular neighborhood $Q$ of $K$ that must be PL-homeomorphic
to the $p$-disc $D^p$. As it can be built of nicely intersecting neighbourhoods of individual
$0$-, $1$-, and $2$-simplices of $K$, we can construct $Q$ explicitly triangulated with $O(N\ln N)$
simplices. The boundary of $Q$ which is a PL- (p-1)-sphere will also be triangulated.
Now we can complete the triangulation of $Q$ to a triangulation of $S^p$ by triangulating the
$p$-disc $S^p\setminus Q$ as the cone over the triangulation of $\partial Q$ (with one new
$0$-simplex). Note that here we loose a control over complexities of (maps of) simplices
in this triangulation (measured as the number of simplices in the corresponding PL-mappings
from the considered simplex to $S^p$), but we do not care about this information. What matters
for us is that the number of simplices in the resulting triangulation will be $O(N\ln N)$,
as claimed.  

\end{proof}

\section{Triangulations} \label{triang}

Our next goal will be to transform each of $\lfloor N^{CN}\rfloor$ balanced finite presentations of the trivial group constructed in Theorem \ref{superexppr}
into triangulations of $S^4$ (or the disc $D^4$) with $O(N\ln N)$ $4$-simplices that are far from each
other in the metric defined as the minimal number of bistellar transformations required
to pass from one triangulation to the other.
 
For this purpose we would like to consider the connected sum of $S^1\times S^3$ (one copy for each generator), triangulate it,
represent the elements of the (free) fundamental
group of this manifold by embedded closed curves, subdivide the triangulation so that these curves
become simplicial curves each having a simplicial neighbourhood simplicially isomorphic to $S^1\times D^3$, remove
each of these neighborhoods,  attach copies of $D^2\times S^2$ along the boundaries of the $S^1\times D^3$, and then extend the triangulations
of $\partial (D^2\times S^2)$ to the triangulations of $D^2\times S^2$.
(This is a standard construction of a manifold with a prescribed fundamental group; cf. \cite{bhp}.)

Now we need to verify that this construction can be performed so that the resulting number of simplices will be $O(N\ln N)$; we will
use the proof of Theorem \ref{graphthm}. The idea is to take an $S^4$ and connect to it copies of $S^3\times S^1$ corresponding to generators of $P'$. Then we need to spell out the relations by some curves. We will run the curves parallel to each other inside the $S^3\times S^1$, in $S^4$ the curves will be placed according to the graph $G(P')$ from Theorem \ref{graphthm}. Below are the details. \par

 To each generator of $P'$ we assign a copy
$R_i$ of $S^3\times S^1$ minus a small $4$-dimensional disc. Denote its boundary $S_i$. We assume that
all $S_i$ are also boundaries of small ``holes" $D_i$ in a copy of $S^4$ ($S_i=\partial D_i$). In $S_i$ we are going to have
$O(1)$ points separated into $O(1)$ groups contained in disjoint small discs (``stations"). These points will be pairwise connected
with each other by arcs inside $R_i$ (which together with any arc connecting the endpoints in $S_i$ will represent the generator
of $\pi_1(R_i)$).  All points in the $j$th station of $S_i$ will be connected with exactly the same number of points in
the $i$th station of $S_j$. We can assume that these arcs run inside a thin cyclinder (``cable") $D^3\times [0,1]$ parallel to each other.
Each cable can be triangulated into $O(1)$ simplices; cables are in the obvious 1-1 correspondence with
edges of the graph $G''(P')$ obtained from $G(P')$ by making the multiplicity of edges and loops equal to one.
Now we observe that if we fill each $S_i$ by a $4$-disc, then these discs and cables can be regarded as a thickening of
a copy of $G''(P')$ in $S^4$; the triangulation of cables can be extended to a triangulation of $S^4\setminus\bigcup D_i$ with
$O(N\ln N)$ simplices exactly as in the proof of Theorem 6.6. This triangulation can be extended to a triangulation
of $S^4$ by adding cones over the triangulations of $S_i$.
Now note that $R_i$ together with $O(1)$ arcs can be triangulated into $O(1)$ simplices; insert this triangulation in one of the $4$-simplices in the cone
over $S_i$ and retriangulate the complement in this simplex into $O(1)$ simplices. Now we need to attach $O(1)$ paths between vertices
of $S_i$ in the triangulation of $S^4\setminus\bigcup D_i$ with the corresponding points in the triangulation of $\partial R_i$.
We are talking about $O(1)$ paths; each of them can pass once through
$O(deg(v_i))$ $4$-dimensional simplices in the triangulation of $D_i$. Here
$v_i$ denotes the vertex of $G''(P')$ corresponding to the $i$th generator of $P'$, and $deg(v_i)$
denotes its number of neighbours in the $1$-skeleton of the triangulation of $S^4$ constructed in the
proof of Theorem 6.6. We retriangulate these simplices in order to make these
$O(1)$ paths simplicial, which will require $O(deg(v_i))$ new simplices.
Since all these degrees sum up to $O(N\ln N)$, the total number of simplices will still be
$O(N\ln N)$.

Now all relators are represented by simple simplicial cycles; removing their tubular neighborhoods and retriangulating increases the number
of simplices by at most a constant factor. Attaching $2$-handles $S^2\times D^2$ increases the number of simplices by a $O(N)$
summand, and we are done.
\par
As the result, we obtained a simplicial complex $T$ with $O(N)$ simplices and a PD-homeomorphism (piecewise differentiable) $f: |T| \to S^4$. There exists the usual PD-homeomorphism $g: \partial \Delta^5 \to S^4$. By the standard theory then there exist a PL-homeomorphism $|T| \to \partial \Delta^5$ (an approximation to $g^{-1} \circ f$). Therefore $T$ is a triangulation of $\partial \Delta^5$, where the initial finite presentation
can be regarded as an ``apparent" finite presentation of the
fundamental group of $T$ in the following sense. Let $N$ denote the number of 
$4$-dimensional simplices of $T$. (Note that here we changed the meaning of
the notation $N$.)

\begin{defn}
Let $T$ be a simplicial complex with $N$ maximal simplices. A finite presentation is
an apparent presentation of $\pi_1(T)$ if it can be connected
by at most $2^{2^N}$ Tietze transformations with one of the following presentations:
Choose a maximal rooted spanning tree of the $1$-skeleton of $T$. Its root will be regarded as a base point.
There is one generator for each edge of the $1$-skeleton not in the tree,
and one (obvious) relation for each $2$-simplex of $T$.
\end{defn}

(We do not really need $2^{2^N}$ in this definition and believe that we could define ``being apparent" as $CN^2$-close,
but prefer to be generous.)
It is not difficult to see that for some constant $C$ two apparent finite presentations of $\pi_1(T)$
can be connected by $3CN^2$ Tietze transformations. Indeed, one needs to check this only for presentations corresponding to different spanning
trees. It is easy to see that one can pass from one spanning tree to another by elementary operations of the following type: Add an edge
not in a tree and remove an adjacent edge in the (unique) resulting cycle. This fact was first observed by L. Lovasz and A. J. Bondy (cf. \cite{lovasz}).
It is easy to prove this fact by induction with respect to the sum of the number of vertices $V$ and the number of edges $E$ of the ambient graph $G$.
(Indeed, take an edge $e$ of $G$. If both trees do not contain $e$, consider $G-e$ and apply the induction assumption. If they both contain
$e$ replace $G$ by $G/e$ and apply the induction assumption. If one of the trees, $T_1$, does not contain $e$ but the other tree, $T_2$, contains $e$, add $e$
to $T_1$ and remove an adjacent edge from the resulting cycle reducing the situation to the case when both trees contain $e$, and the induction
assumption applies.) This proof leads to an upper bound $V+E$ for the number of the required operations, that in our case is $\leq const N$.
It remains only to check what happens with the finite presentation of $\pi_1(T)$ during one such elementary operation on
spanning trees and to verify that the change of the presentation can be  described as the result of at most $const N$ Tietze operations. This fact
easily follows
from the observation that each generator of $\pi_1(T)$ in the new presentations corresponds either to an old generators
or to the result of conjugation of one old generator by another.
 
Recall that a bistellar transformation (a.k.a. a Pachner move)
of an $n$-dimensional simplicial complex replaces a simplicial subcomplex that is composed of between $1$ and $n+1$
$n$-dimensional simplices and is isomorphic to a part of the boundary $\partial \Delta^{n+1}$ of an $(n+1)$-dimensional simplex by the
complementary part of $\partial \Delta^{n+1}$. Finally we will need the following lemma.

\begin{lem} If simplicial complexes $T_1, T_2$ are related by just one bistellar transformation, then
there are spanning trees for $T_1$ and $T_2$ such that the presentations of $\pi_1(T_1)=\pi_1(T_2)$ corresponding to those trees can be transformed one into the other by just $const$ Tietze transformations.
\end{lem}
\begin{proof}
We can first consider the subcomplexes of $T_1$ and $T_2$ (call them $R_1, R_2$) involved in the transformation, and take some spanning trees for them which coincide on the boundary. Since they have the same $\pi_1$, the presentations corresponding to the trees can be connected by Tietze transformations. There is a finite number (for a fixed dimension) of possible bistellar transformations, therefore there is a bound on how many Tietze transformations we need. Now, we can extend the trees to spanning trees of the whole of $T_1$ and $T_2$ (extend them in the same way). The corresponding presentations of $\pi_1(T_1)=\pi_1(T_2)$ will differ by the same Tietze transformations as the presentations for $\pi_1(R_1)=\pi_1(R_2)$.
\end{proof}

Combining all these observations with Theorem B (Theorem \ref{ch5:mainthmk}) we see that
we obtain the following theorem in the case when $M^4=S^4$. 

\begin{thm}
Let $M^4$ be a simplicial closed manifold or a compact manifold with boundary of dimension $4$. There exists a constant $C>1$ such that for
each $m$ all sufficiently large $N$ there exist more than $C^N$ triangulations of $M^4$ with
at most $N$ $4$-dimensional simplices such that no pair of these triangulation can be transformed one into the other by less than $\exp_m(N)$ bistellar
transformations.
\end{thm}

In order to obtain the general case we take any fixed triangulation of $M^4$ and take its simplicial connected sum with the constructed
triangulations of $S^4$. The result will immediately follow when $M^4$ is simply-connected. If it is not simply-connected, we will need to verify that
considered balanced presentations of the trivial group will remain
very distant from each other even after we will take the free product
with an apparent finite presentation of $\pi_1(M^4)$ constructed using
the chosen triangulation of $M^4$. This follows from Theorem \ref{nontrivial}. That finishes the proof of Theorem A.

Theorem C.1 follows more directly. Let $a$ be a generator of the presentation used to construct $K_1$ realized as a simplicial path on $K_1$ in the obvious way. Then the existence of the homotopy $H_1$ gives us that the area $F_1(a)a^{-1}$ is less than the witness complexity (area is measured in an apparent presentation of ${K}_1$). Therefore we obtain an $(L,L)$-effective isomorphism (see Definition \ref{effectiveiso}), where $L$ is equal to the witness complexity, between apparent presentations of $K_i$.

Finally, we would like to sketch the proof of Theorems A.1 and A.1.1.

\begin{proof}

We start with $M=S^4$. We are not going to use the constructions of Section 6 and the beginning of this section. Instead we use the exponential number of finite presentations from Theorem B with $x=O(N)$. Note that for all sufficiently large $N$ we can choose them so that they are not pairwise $(L_1, L_2)$-isomorphic
 (in the sense of Definition \ref{effectiveiso}) for all values of $L_1$, $L_2$ up to the $\exp_{m+10}(N)$. Recall that these finite presentations have four generators
and four relators.

At the first stage of this proof for each of these finite presentations we construct the corresponding Riemannian metric on $S^4$.
We take the connected sum of four copies of handles $S^1\times S^3$. We need to realize the relators by simple closed curves with the injectivity radius of the normal exponential map not less than $1$.
Then we will remove tubular neighborhoods of these curves, insert the $2$-handles killing the relators and smooth out the corners. If at the end the upper bound for the absolute value of the sectional curvature will not be
$1$ but $const$, we can always rescale the metric to the desired upper bound without significantly affecting the other quantities.  In order to do that we can choose
Euclidean product metrics on all spheres/handles/discs involved in the construction with all linear sizes $O(N)$.  But, in fact, we can ensure that
the diameter of the resulting Riemannian $S^4$ will  be $O(\ln N)$ while the volume will be $O(N^3)$. To do that we attach the $1$-handles $S^1\times S^3$ to a Riemannian sphere $S^4$ glued out of two hyperbolic discs of radius $O(\ln N)$ glued along the common boundary (and then smoothed out). The closed curves corresponding
to relators go between each pair of handles in ``bunches" almost parallel to each other. But they need to wiggle inside the handles as they need to arrive
to a right position on the other side. As we need to keep the distance $O(1)$ between them, it is obviously sufficient
to choose the length of $S^1$ in at least some of the handles
as $\sim N^2$, and the
diameter will be $\sim N^2$ as well. (Note that transversal $S^3$ must also have volume $\sim N$, but we can glue it out of $2$ copies of a hyperbolic
$3$-disc of radius $O(\ln N)$, so the problem with the diameter does not arise.) 
Yet we are going to kill the relators by attaching $2$-handles with hyperbolic $2$-discs at their axes.
The circumference of these discs is $O(N^2)$, so the radius $O(\ln N)$ would suffice. The dimensions of perpendicular $S^2$s will be 
$O(1)$. Now the diameters of $1$-handles will
become $O(\ln N)$ as we now can move from a point to a point through the $2$-handles. As the volume $x$ will be bounded by $const N^3$, we will be ``entitled"
to lower bound of the form $\exp(const\ x^{1\over 3})$ for the number of the pairwise distant metrics $\mu_i$. It is now clear that
the constructed metrics satisfy property 1 in the text of Theorem A.1 and conditions of Theorem A.1.1 for $c_0={1\over 3}$.

We must prove that the constructed Riemannian metrics are not pairwise $\exp_m(x)$-bi-Lipschitz homeomorphic. We would like to define the concept of apparent finite presentations.
We will need this concept for Riemannian $4$-spheres $M$ from $Al_1(S^4)$ that also satisfy a positive lower
bound for the volume and an upper bound for the diameter as in Theorem A.1.1 (or Condition 1 of Theorem A.1).
Assume that $\epsilon_0$ is less than ${1\over 1000}$ of the standard lower bound for the convexity
radius for manifolds satisfying curvature, diameter, volume constraints as in Theorem A.1.1 (or Condition 1 of Theorem A.1).
(This $\epsilon_0$ behaves $\sim {v\over\exp (3D)}$, where $D$ is an upper bound for the diameter of the manifold).
We are going to $\epsilon_0/10$-approximate $M$ by a metric simplicial $2$-complex $K$ (in the Gromov-Hausdorff metric), so
that all closed curves in $K$ of length $<10\epsilon_0$ are contractible with increase of length only by at most a constant factor
during a contracting homotopy. We would like to call a presentation of $\pi_1(K)$ obtained using any spanning tree of its $1$-skeleton
``an apparent finite presentation of $\pi_1$" of $M$. However, as we are going to have a lot of ambiguity in constructing $K$, and, then,
choosing  the spanning tree, we are going to say that all finite presentations that are $(L_1,L_2)$-effectively isomorphic to 
the first chosen finite presentation of $\pi_1(K)$ (for one possible choice of $K$) are also apparent. Here we generously allow
$L_1$ and $L_2$ to grow triply exponentially with $x$ (the ``$x$'' from Theorem A.1.1). In this way, we are going to eliminate
the dependence on all choices that one needs to make while defining a finite presentation of $\pi_1(M)$. Also, it will be clear from
the construction of $K$ below that   
the finite presentations $\mu_v$ used to define Riemannian $4$-spheres $M=M_{\mu_v}$ in the previous paragraph will also be apparent
finite presentations of $\pi_1$ of $M_{\mu_v}$.

A key fact behind this idea is that two Gromov-Hausdorff close compact length spaces with the same  contractibility function have effectively
isomorphic
fundamental groups (cf. \cite {petersen}). (Note that we have much more than a linear filling function here -- we have a control over
the Lipschitz constant of a homotopy contracting short curves on $M$ and, as we will see, the same will be true for $K$.) The idea
is that if $X_1, X_2$ are two $\delta$-close spaces, where all closed curves of length $<10\delta$ can be ``easily" contracted we can construct
isomorphisms between $\pi_1(X_1)$ and $\pi_1(X_2)$ proceeding as follows: Represent an element of $\pi_1(X_1)$ by a closed curve $\gamma$.
Subdivide it into $\bar N$ arcs of equal small length ( here $\bar N$ is a fixed large parameter). We can associate points $x_i\in\gamma$ of the subdivision
to the nearest points $x_i'$ of $X_2$ (here we assume that we have chosen a metric on the disjoint connected sum of $X_1$ and $X_2$ realizing the
Gromov-Hausdorff distance). We call $x_i'$ a transfer of $x_i$. Pairwise connect $x_i'$ by geodesic arcs to obtain a ``transfer" $\gamma'$ of $\gamma$ on $X_2$.
Note that when we take the transfer of $\gamma'$ back to $X_1$, we obtain a closed curve $\gamma''$ formed by arcs between $x_i''$ and $x_{i+1}''$
obtained by transfer of $x_i'$ back to $X_1$. Now a construction of a homotopy between $\gamma$ and $\gamma''$ effectively reduces to contracting
all quadrilaterals $x_ix_i'' x_{i+1}x_{i+1}''$ that have length $<9\delta$, if $\bar N$ is sufficiently large. If we take $X_1 = K_1, X_2 = K_2$ for $2$-complexes $K_i$ with the properties described in the previous paragraph, then the presentations of $\pi_1$ of these complexes (and therefore of $M$) obtained from spanning trees will be effectively isomorphic.

Now we are going to explain how to construct $K$. Take $\epsilon=\epsilon_0/100,\ \delta={\epsilon^2\over 10000diam (M)}$.
Consider a $\delta$-net in $M$ such that metric balls of radius $\delta/4$ centered at points of the net do not intersect.
Croke's inequality will imply a positive lower bound $const\delta^4$ for the volumes of pairwise non-intersecting balls, and we obtain
an upper bound ${\exp(const\ diam(M))\over v^{const}}$ for the number of points in the net. Construct
a metric graph (the $1$-skeleton of $K$) by connecting each pair of $\epsilon$-close points by a $1$-simplex. The length of this simplex
is defined as the distance between the vertices in $M$. As in the proof of Lemma 7.5.5 in \cite {burago} we see that the resulting metric graph $\epsilon\over 1000$-approximates $M$. Now we are going to fill
all triangles in the graph of perimeter $\leq\epsilon/3$ by $2$-simplices. The metric on each $2$-simplex is defined as the metric on a round
hemisphere of radius equal to ${P\over 2\pi}$, where $P$ is the perimeter of the filled triangle. In this way when we are connecting
two points on the boundary of one of these triangles by a geodesic, there will be no
shortcuts through the interior of the $2$-simplex. Therefore, the distance between any two points of the $1$-skeleton of the resulting
metric $2$-complex $K$ will be the same as the distance between these points in the $1$-skeleton. On the other hand $K$ will be $\epsilon/12$-close
to its $1$-skeleton, and therefore $\epsilon/10$-close to $M$.
We claim that not only all simplicial closed curves of length $<\epsilon /2$ are contractible, but all closed curves
of length $\leq 1000\epsilon= 10\epsilon_0$, as we claimed. It is sufficient to check this for simplicial curves of
length $\leq 2000\epsilon$. In order to contract a simplicial curve $\gamma$ on $K$ of length $\leq 2000\epsilon$
we first discretize it and transfer it to $M$ ( as it was described above). Then we
contract the result of the transfer in $M$ by a length non-increasing homotopy. ( Here we
are using the fact that the length of the transfer of $\gamma$ to $M$ will still be less than the convexity radius of $M$.)
Then we discretize this homotopy, and transfer the discrete set of close curves in this homotopy back to the $1$-skeleton of $K$ that
is ${\epsilon\over 1000}$-close to $M$. Now we need to
connect by homotopies in $K$ $\gamma$ with its image after two transfers, and the transfers of each pair of consecutive curves.
In order to do this we need to contract simplicial quadrilaterals of length $<<\epsilon$, and, in particular, $<\epsilon/6$. We can
represent each of these two quadrilaterals as
the union of two simplicial triangles that will still be sufficiently short to be filled by $2$-simplices in $K$.
(Note that when we need to contract a closed curve in $K$ that is not necessarily simplicial, we can start from sliding all arcs in
the interiors of $2$-simplices to the boundary by length non-increasing homotopies reducing the situation to the case of simplicial closed curves.)

The same idea of transfer of closed curves and homotopies contracting closed curves can be easily adapted to show that
presentations corresponding to different choices of spanning trees in the $1$-skeleton of $K$ or different choices of nets
in the construction of $K$ lead to $(L_1,L_2)$-effectively isomorphic presentations for acceptable for us values of $L_1, L_2$. The proof of the first fact
can also be done along the same lines as the proof of a similar assertion in the PL-case that was considered above in this section. The second assertion
follows from the fact that $K_1$ and $K_2$ that correspond to different choices of nets will still be Gromov-Hausdorff close to $M$, and,
therefore, to each other. Further, if $M_1$ and $M_2$ are $\epsilon_0$-close, then an apparent presentation of the
fundamental group of one of them will be an apparent finite presentation of the other.

If we would like to extend this observation
for $M_1$ and $M_2$ that can be connected by a (not very long) sequence of short jumps (as in condition 2 of Theorem A.1), then we will need to
(controllably) increase $L_1$ and $L_2$ in the condition ``being $(L_1, L_2)$-effectively isomorphic".
An effective proof of the precompactness of sublevel
sets of diameter on $Al_1(M)$ implies that if two points in the sublevel set $diam^{-1}((0,x])$ can be connected by a finite sequence of short jumps, they can be connected
by a sequence of short jumps of length that does not exceed $\exp(\exp(const\ x))$. Such a sequence of short jumps
would add two extra exponentiations to our upper bounds for $L_1, L_2$. If the jumps go through Riemannian metrics with much higher
values of the diameter, say, up to $\exp_m(x)$, and we correspondingly adjust $\epsilon_0$ (which now will behave as 
$const \exp(-3\exp_m(x))$), then $L_1, L_2$ will be bounded above by $\exp_{m+const}(x)$. 
Thus, if two Riemannian metrics $M_{\mu_{v_1}}$, $M_{\mu_{v_2}}$ can be connected by a sequence of short jumps as in Theorem A.1
hen the corresponding presentations $\mu_{v_1}$, $\mu_{v_2}$ are $(\exp_{m+const}(x), \exp_{m+const}(x))$-effectively
isomorphic (where one can safely choose $const=10$), which according to our assumption can happen only if $v_1=v_2$.
Thus, we directly obtain Theorem A.1 (bypassing Theorem A.1.1).

To prove Theorem A.1.1 we will choose $\epsilon_0={1\over\exp_{m+4}(x)}$. If the constructed Riemannian spheres are $\exp_m(x)$-bi-Lipschitz homeomorphic,
then the apparent finite presentations of their fundamental groups will be $(\exp_{m+1}(x), \exp_{m+1}(x))$-effectively isomorphic, and we obtain a contradiction proving the theorem. (A minor technical difficulty here is that when we realize a generator of the fundamental
group of one of these spheres by a broken geodesic, and map it into the other sphere using the Lipschitz homeomorphism $f$, we end up
with a curve that needs to be approximated by a broken geodesic. But our choice of $\epsilon_0$ ensures that the image of this broken
geodesic under $f^{-1}$ will be close to the broken geodesic representing the original generator of the fundamental group, and these
curves can be connected by an obvious homotopy that almost does not increase the length.)

Similarly to the proof of Theorem A, the case of general $M$ can be proven by forming a Riemannian connected sum of a fixed Riemannian metric on $M$ and the constructed very distant metrics on $S^4$ and using Theorem \ref{nontrivial}.
\end{proof}

\par\noindent
{\bf Acknowledgement:} This work was partially supported by NSERC Discovery grant 155879-12 of Alexander Nabutovsky.
The authors would like to thank Misha Gromov for his suggestion to recast Theorem C as an assertion about geometry of 
contractible $2$-complexes. (Remark BC and Theorem C.1 are our implementations of this suggestion.)

\bibliographystyle{alpha}

\bibliography{bibliography}

\end{document}